\documentclass[11pt,reqno]{amsproc}
\usepackage{tipa}
\usepackage{bbm}
\usepackage{mathrsfs}
\usepackage{stmaryrd}
\usepackage{cases}
\usepackage{amssymb}
\usepackage{amsmath}
\usepackage{amsfonts}
\usepackage{graphicx}
\usepackage{amsmath,amstext,amsbsy,amssymb}
\usepackage[table]{xcolor}
\usepackage{multirow}
\usepackage{cases}
\usepackage{color}
\usepackage{amsthm}

\newtheorem{theorem}{Theorem}[section]
\newtheorem{lemma}[theorem]{Lemma}

\newtheorem{proposition}{Proposition}[section]
\newtheorem{corollary}[theorem]{Corollary}

\theoremstyle{definition}
\newtheorem{definition}[theorem]{Definition}

\theoremstyle{remark}
\newtheorem*{remark}{Remark}
\numberwithin{equation}{section}

\newcommand{\Y}{\mathrm{Y}}
\newcommand{\DY}{\mathrm{DY}}
\newcommand{\ben}{\begin{equation*}}
\newcommand{\een}{\end{equation*}}
\newcommand{\ts}{\,}
\newcommand{\tss}{\hspace{1pt}}
\newcommand{\beq}{\begin{equation}}
\newcommand{\eeq}{\end{equation}}
\newcommand{\wh}{\widehat}
\newcommand{\Ec}{\mathcal{E}}
\newcommand{\Fc}{\mathcal{F}}
\newcommand{\Hc}{\mathcal{H}}
\newcommand{\eb}{\bar{e}}
\newcommand{\gr}{ {\rm gr}\ts}
\newcommand{\Ac}{\mathcal{A}}

\begin{document}
\title[Center of the Yangian double in type A]
{Center of the Yangian double in type A}
\author{Fan Yang}
\address{Department of Mathematics,
Jiaying University, Meizhou, Guangdong 514000, China}
\email{$1329491781@qq.com$}

\author{Naihuan Jing$^{\dagger}$}
\address{Department of Mathematics,
   North Carolina State University,
   Raleigh, NC 27695, USA}
\email{jing@ncsu.edu}

\keywords{Yangian double, Harish-Chandra homomorphism, Wakimoto module} 
\thanks{$\dagger$Corresponding author: Naihuan Jing}

\begin{abstract}
We prove the R-matrix and Drinfeld presentations of the Yangian double in type A are isomorphic.
The central elements of the
completed Yangian double in type A at the critical level are constructed. The images of these elements under a Harish-Chandra-type
homomorphism are calculated by applying a version of the Poincar\'{e}-Birkhoff-Witt theorem for the R-matrix
presentation. These images coincide with the eigenvalues of the
central elements in the Wakimoto modules.
\end{abstract}
\maketitle
\section{Introduction}

During the last few decades, quantum groups introduced by Drinfeld \cite{Dr2} and Jimbo \cite{Jb} have been studied extensively
and played important roles in many branches of mathematics and physics. The Yangian $\Y_{h}(\mathfrak g)$
is one of the most important examples in the theory of quantum groups. The Hopf algebra first appeared
in the framework of the inverse scattering method introduced by
Faddeev, Sklyanin and their colleagues, see \cite{FT},\cite{FST},\cite{FRT}.
As is well-known, the Yangian $\Y_{h}(\mathfrak{g})$ is a canonical
deformation of the universal enveloping algebra $\mathrm{U}(\mathfrak
g[t])$. The Yangian double, which is the quantum double \cite{Dr1}
of the Yangian $\Y_{h}(\mathfrak{g})$, has nice applications
in massive field theory \cite{BL},\cite{LS},\cite{S}. In these works, the Yangian double
$\DY_{h}(\mathfrak{g})$ without central extension was discussed
for $\mathfrak g=\mathfrak {sl}_{2}$ \cite{K}. As for the general case,
Iohara \cite{I} defined the Yangian double
$\DY_{h}(\mathfrak{g})$ with a central extension for $\mathfrak
g=\mathfrak {gl}_{n}, \mathfrak {sl}_{n}$.

The purpose of this paper is to give explicit formulas for generating central elements
of the completed Yangian double $\DY_{h}(\mathfrak{gl}_n)_{cr}$ at
the critical level (Theorem 4.4).  The formulas are given with the help of the RLL-presentation of $\DY_{h}(\mathfrak{gl}_n)$. We use
a version of the Poincar\'{e}-Birkhoff-Witt theorem for this
presentation to introduce an analog of the Harish-Chandra
homomorphism and calculate the image of the central
elements, and these central elements are constructed as
certain higher Casimir operators (cf. \cite{CM}). Moreover, we construct an isomorphism between the R-matrix and Drinfeld realization of the Yangian double $\DY_{h}(\mathfrak{gl}_n)$, i.e. the Yangian analog of the Ding-Frenkel isomorphism in quantum affine algebras \cite{DF}.
Then we apply the isomorphism to calculate the eigenvalues of the central
elements in the Wakimoto modules over $\DY_{h}(\mathfrak{gl}_n)_{cr}$.
As an application of Theorem 4.4, we give explicit formulas for invariants of
the vacuum module $V_{h}(\mathfrak{gl}_n)$ over the Yangian double.
The invariants are obtained by the action of the central elements on
the vacuum vector.

The center of the quantum affine algebra at the critical level has been studied extensively \cite{DE, FJ}. It is thus expected that similar results
hold for the Yangian double, and we intend to give detailed computation for the image of Harish-Chandra homomorphism.
The Yangian double in type A at the critical level admits a quantum vertex algebra structure \cite{ji:cen}. In the BCD cases we expect there is a similar
quantum vertex algebra structure as results of this paper indicate.

The layout of the paper is as follows. In section \ref{s:Drin} we give detailed construction of the Yangian double $DY_h(\mathfrak{gl}_n)$
in the FRT-formulism and study the quantum root generators. This result should be
known to the experts but we include the details as we could not find it in the literature, though the similar construction is given for
type BCD in \cite{JYL}.
In section \ref{s:drins} we pass it down to the Drinfeld realization of the double Yangian $DY_h(\mathfrak{sl}_n)$ using a different method from
that of the quantum affine algebra. The center of the completed algebra $DY_h(\mathfrak{gl}_n)$ at the critical level is discussed in section \ref{s:cent}. In this section, we use a slightly different definition of $DY_h(\mathfrak{gl}_n)$ by normalizing the Yang R-matrix.
In section \ref{s:HS-hom} we define the Harish-Chandra homomorphism for the quantum algebra $DY_h(\mathfrak{gl}_n)_{cr}$, and then in section \ref{s:eigenvalue} we determine the eigenvalues of the Harish-Chandra operator at the critical level and find that the eigenvalues are also expressed
as some shifted elementary symmetric functions (cf. \cite{Mo3}).

\section{Drinfeld generators of the Yangian double}\label{s:Drin}
Let $V=\mathbb{C}^n$, and $P$  the permutation operator of $V\otimes
V$ such that
$$Pv\otimes w=w\otimes v,  \qquad v,w\in
V.
$$

Let $h$ be a parameter, and let $\bar{R}(u)$ be Yang's R-matrix:
\begin{equation}\label{1}
\bar R(u)=I+\frac{h}{u}P\in End(V\otimes V).
\end{equation}
where $\bar R(u)$ is expanded in negative powers of $u$ or positive
powers of $h$. The matrix $\bar R(x)$ satisfies the Yang-Baxter
equation on $V^{\otimes 3}$ with the unitary condition:
\begin{eqnarray}
\bar R_{12}(u-v)\bar R_{13}(u)\bar R_{23}(v)&=&\bar R_{23}(v)\bar R_{13}(u)\bar R_{12}(u-v),\\
\bar R_{12}(u)\bar R_{21}(-u)&=&\frac{u^{2}-h^{2}}{u^{2}}I.
\end{eqnarray}
Here for any $A=\sum a_{(1)}\otimes a_{(2)}\in End(V^{\otimes 2})$, one denotes 
that $A_{12}=A\otimes I, A_{23}=I\otimes A$, and $A_{13}=\sum
a_{(1)}\otimes I\otimes a_{(2)}\in End(V^{\otimes 3})$.

Let $\mathcal A$ be the ring $\mathbb C[[h]]$ in the $h$-adic
topology.

\begin{definition}\label{e:def-hopf}
The Yangian double $\DY_{h}(\mathfrak{gl}_n)$ is the associative
algebra over $\mathcal {A}$ generated by $\{l_{ij}^{(k)}\mid 1\leq
i,j\leq n,~k\in\mathbb{Z}\}$ and $c$. The relations are defined in
terms of the generating series
\begin{align*}
l_{ij}^{+}(u)&=\delta_{ij}-h\sum_{k\in\mathbb{Z}_{\geq 0}}l_{ij}^{(k)}u^{-k-1},\\
l_{ij}^{-}(u)&=\delta_{ij}+h\sum_{k\in\mathbb{Z}_{<
0}}l_{ij}^{(k)}u^{-k-1},
\end{align*}
and $l_{ij}^{\pm} (u)$ are in the matrix form:
\begin{align*}
L^{\pm}(u)=(l_{ij}^{\pm}(u))_{1\leq i,j\leq n}.
\end{align*}
The defining relations are given as follows.
\begin{align}
&[L^{\pm}(u),c]=0,\\ \label{e:2.5} 
&\bar R(u-v)L_1^{\pm}(u)L_2^{\pm}(v)=L_2^{\pm}(v)L_1^{\pm}(u)\bar R(u-v),\\ \label{e:2.6}  
& \bar
R(u-v-\frac{1}{2}hc)L_1^{+}(u)L_2^{-}(v)=L_2^{-}(v)L_1^{+}(u)\bar
R(u-v+\frac{1}{2}hc),
\end{align}
where the R-matrix is viewed as power series in $u^{-1}$ (via appropriate $h$-adic topology) and we have adopted the convention
\begin{align*}
L_{1}^{\pm}(u)&=L^{\pm}(u)\otimes id, \\
L_{2}^{\pm}(u)&=id\otimes L^{\pm}(u).
\end{align*}
The algebra $\DY_{h}(\mathfrak{gl}_n)$ has a Hopf algebra structure
given by
\begin{eqnarray*}
\triangle(L^{\pm}(u))&=&\sum_{k=1}^{n}L_{kj}^{\pm}(u\pm\frac{1}{4}hc_2)\otimes L_{ik}^{\pm}(u\mp\frac{1}{4}hc_1),\\
\triangle(c)&=&c\otimes 1+1\otimes c,\\
\varepsilon(L^{\pm}(u))&=&I, \qquad \varepsilon(c)=0,\\
S(^{t}L^{\pm}(u))&=&[^{t}L^{\pm}(u)]^{-1},\qquad S(c)=-c,
\end{eqnarray*}
where $c_1=c\otimes 1 ~and~ c_2=1\otimes c$.
\end{definition}

\begin{definition}
For $1\leq p,q\leq n$, we define the submatrices $L_{p,q}^{\pm}(u)$
of $L^{\pm}(u)$ as follows.
\begin{gather}
p=q,\quad L_{p,p}^{\pm}(u)=(l_{ij}(u))_{1\leq i,j\leq p}\\
\nonumber\\
p<q,\quad L_{p,q}^{\pm}(u)=\begin{pmatrix}
l_{11}^{\pm}(u)&\ldots&l_{1,p-1}^{\pm}(u)&l_{1,q}^{\pm}(u)\\
\vdots&&\vdots&\vdots\\
l_{p-1,1}^{\pm}(u)&\ldots&l_{p-1,p-1}^{\pm}(u)&l_{p-1,q}^{\pm}(u)\\
l_{p,1}^{\pm}(u)&\ldots&l_{p,p-1}^{\pm}(u)&l_{p,q}^{\pm}(u)
\end{pmatrix}\\
\nonumber\\
p>q,\quad L_{p,q}^{\pm}(u)=\begin{pmatrix}
l_{11}^{\pm}(u)&\ldots&l_{1,q-1}^{\pm}(u)&l_{1,q}^{\pm}(u)\\
\vdots&&\vdots&\vdots\\
l_{q-1,1}^{\pm}(u)&\ldots&l_{q-1,q-1}^{\pm}(u)&l_{q-1,q}^{\pm}(u)\\
l_{p,1}^{\pm}(u)&\ldots&l_{p,q-1}^{\pm}(u)&l_{p,q}^{\pm}(u)
\end{pmatrix}
\end{gather}
\end{definition}

Then we can give the Gauss decomposition of $L^{\pm}(u)$ in the following result.

\begin{theorem}
$L^{\pm}(u)$ have the following unique decompositions:
\begin{align*}
L^{\pm}(u)=&F^{\pm}(u)H^{\pm}(u)E^{\pm}(u)\\
=&\begin{pmatrix}
1&&&0\\f_{2,1}^{\pm}(u)&\ddots\\\vdots&&\ddots\\f_{n,1}^{\pm}(u)&\ldots&f_{n,n-1}^{\pm}(u)&1
\end{pmatrix}
\begin{pmatrix}
k_{1}^{\pm}(u)&&&0\\&\ddots\\&&\ddots\\0&&&k_{n}^{\pm}(u)
\end{pmatrix}\nonumber\\
&\begin{pmatrix}
1&e_{1,2}^{\pm}(u)&\ldots&e_{1,n}^{\pm}(u)\\&\ddots\\&&\ddots&e_{n-1,n}^{\pm}(u)\\0&&&1
\end{pmatrix}\nonumber
\end{align*}
\end{theorem}

To prove this theorem, we only need to show that each component
$f_{p,q}^{\pm}(u)$, $k_{p}^{\pm}(u)$ or $e_{p,q}^{\pm}(u)$ is
well-defined in $\DY_{h}(\mathfrak{gl}_n)[[u^{\mp1}]]$. In
fact, we have the following explicit formulas of these elements in
terms of quantum minors. Then the elements are well-defined
since the algebra
$\DY_{h}(\mathfrak{gl}_n)$ is $h$-adically completed.

The following lemma is proved in \cite{I}. Recall that the (column) quantum determinant of a matrix $M=(m_{ij}(x))_{n\times n}$ is defined
by
$$qdet(M)=\sum_{\sigma\in \mathfrak S_n}sgn(\sigma)m_{\sigma_11}(x)m_{\sigma_22}(x+h)\cdots m_{\sigma_nn}(x+(n-1)h).
$$

\begin{lemma}
We have the following expansions for the quantum determinants
\begin{align}\label{e:2.10}
qdetL^{\pm}(u)=k_{1}^{\pm}(u)k_{2}^{\pm}(u+h)\cdots
k_{n}^{\pm}(u+(n-1)h).
\end{align}
\end{lemma}

Now set
\begin{align*}
&X_{i}^{-}(u)=f_{i+1,i}^{+}(u+\frac{1}{4}hc)-f_{i+1,i}^{-}(u-\frac{1}{4}hc),\\
&X_{i}^{+}(u)=e_{i,i+1}^{+}(u-\frac{1}{4}hc)-e_{i,i+1}^{-}(u+\frac{1}{4}hc).
\end{align*}
To decompose $\DY_{h}(\mathfrak{gl}_n)$ into a product of two subalgebras, we introduce the
following currents:
\begin{align*}
&H_{i}^{\pm}(u)=k_{i+1}^{\pm}(u+\frac{1}{2}hi)k_{i}^{\pm}(u+\frac{1}{2}hi)^{-1},K^{\pm}(u)=\prod_{i=1}^{n}k_{i}^{\pm}(u+(i-\frac{n+1}{2})h),\\
&E_{i}(u)=\frac{1}{h}X_{i}^{+}(u+\frac{1}{2}hi),F_{i}(u)=\frac{1}{h}X_{i}^{-}(u+\frac{1}{2}hi).
\end{align*}
We define $\DY_{h}(\mathfrak{sl}_n)$ to be the subalgebra of
$\DY_{h}(\mathfrak{gl}_n)$ generated by
$H_{i}^{\pm}(u),$\\$E_{i}(u),F_{i}(u)$ and $c$. Note that $\DY_{h}(\mathfrak{sl}_n)$ inherits the Hopf algebra structure (cf. \cite{C}) from that of $\DY_{h}(\mathfrak{gl}_n)$.
 The Heisenberg
subalgebra of $\DY_{h}(\mathfrak{gl}_n)$ generated by $K^{\pm}(u)$
commutes with each element of $\DY_{h}(\mathfrak{sl}_n)$. Next we will write down the relations between the currents
$k_{i}^{\pm}(u),X_{i}^{+}(u)$ and $X_{i}^{-}(u)$.

\begin{theorem}
The following relations hold in the algebra $\DY_{h}(\mathfrak{gl}_n)$:
$$k_{i}^{\pm}(u)k_{j}^{\pm}(v)=k_{j}^{\pm}(v)k_{i}^{\pm}(u)$$
$$\frac{u_{-}-v_{+}+h}{u_{-}-v_{+}}k_{i}^{+}(u)k_{i}^{-}(v)=\frac{u_{+}-v_{-}+h}{u_{+}-v_{-}}k_{i}^{-}(v)k_{i}^{+}(u)$$
$$k_{i}^{+}(u)k_{j}^{-}(v)=k_{j}^{-}(v)k_{i}^{+}(u)(i<j)$$
$$\frac{(u_{+}-v_{-})^{2}}{(u_{+}-v_{-})^{2}-h^{2}}k_{i}^{-}(u)k_{j}^{+}(v)=\frac{(u_{-}-v_{+})^{2}}{(u_{-}-v_{+})^{2}-h^{2}}k_{j}^{+}(v)k_{i}^{-}(u)(i<j)$$
$$k_{i}^{\pm}(u)^{-1}X_{i}^{+}(v)k_{i}^{\pm}(u)=\frac{u_{\pm}-v+h}{u_{\pm}-v}X_{i}^{+}(v)$$
$$k_{i}^{\pm}(u)X_{i}^{-}(v)k_{i}^{\pm}(u)^{-1}=\frac{u_{\mp}-v+h}{u_{\mp}-v}X_{i}^{-}(v)$$
$$k_{i+1}^{\pm}(u)^{-1}X_{i}^{+}(v)k_{i+1}^{\pm}(u)=\frac{u_{\pm}-v-h}{u_{\pm}-v}X_{i}^{+}(v)$$
$$k_{i+1}^{\pm}(u)X_{i}^{-}(v)k_{i+1}^{\pm}(u)^{-1}=\frac{u_{\mp}-v+h}{u_{\mp}-v}X_{i}^{-}(v)$$
$$k_{j}^{\pm}(u)^{-1}X_{i}^{+}(v)k_{j}^{\pm}(u)=X_{i}^{+}(v),k_{j}^{\pm}(u)X_{i}^{-}(v)k_{j}^{\pm}(u)^{-1}=X_{i}^{-}(v)
\quad (j\neq i,i+1)$$
$$(u-v\mp h)X_{i}^{\pm}(u)X_{i}^{\pm}(v)=(u-v\pm
h)X_{i}^{\pm}(v)X_{i}^{\pm}(u)$$
$$(u-v+h)X_{i}^{+}(u)X_{i+1}^{+}(v)=(u-v)X_{i+1}^{+}(v)X_{i}^{+}(u)$$
$$(u-v)X_{i}^{-}(u)X_{i+1}^{-}(v)=(u-v+h)X_{i+1}^{-}(v)X_{i}^{-}(u)$$
\begin{align*}
&X_{i}^{\pm}(u_{1})X_{i}^{\pm}(u_{2})X_{j}^{\pm}(v)-2X_{i}^{\pm}(u_{1})X_{j}^{\pm}(v)X_{i}^{\pm}(u_{2})+X_{j}^{\pm}(v)X_{i}^{\pm}(u_{1})X_{i}^{\pm}(u_{2})\\
&+\{u_{1}\leftrightarrow u_{2}\}=0 \quad if~ |i-j|=1
\end{align*}
$$X_{i}^{\pm}(u)X_{j}^{\pm}(v)=X_{j}^{\pm}(v)X_{i}^{\pm}(u)\quad if~
|i-j|>1$$
\begin{align*}
&[X_{i}^{+}(u),X_{j}^{-}(v)]=h\delta_{ij}\{\delta(u_{-}-v_{+})k_{i+1}^{+}(u_{-})k_{i}^{+}(u_{-})^{-1}-\delta(u_{+}-v_{-})k_{i+1}^{-}(v_{-})\\
&k_{i}^{-}(v_{-})^{-1}\} \quad
where ~\delta(u-v)=\Sigma_{k\in\mathbb{Z}}u^{-k-1}v^{k}.
\end{align*}
\end{theorem}

\begin{remark}
The relations in Theorem 2.5 are similar to those from Iohara's paper \cite{I}, though Iohara's paper did not contain its proof. We also note that the R-matrix $\bar R(u)$ in our paper is a bit different from the R-matrix $R(u)$ used in \cite{I}.
\end{remark}

\begin{proof}
Note that
\begin{align}
X_{i}^{+}(u)=e_{i,i+1}^{+}(u_{-})-e_{i,i+1}^{-}(u_{+}),X_{i}^{-}(u)=f_{i+1,i}^{+}(u_{+})-f_{i+1,i}^{-}(u_{-}),
\end{align}
where $u_{\pm}=u\pm \frac{1}{4}hc$. In the rest of the paper, we will use
$f_{i}^{\pm}(u)$ to denote $f_{i+1,i}^{\pm}(u)$ and $e_{i}^{\pm}(u)$
to denote $e_{i,i+1}^{\pm}(u)$. Since $k_{i}^{\pm}(u)$ are
invertible, we can express the elements $e_{i,j}^{\pm}(u),f_{j,i}^{\pm}(u)$ and
$k_{i}^{\pm}(u) (i<j)$ in terms of the matrix
coefficients of $L^{\pm}(u)$ in a unique manner.

We can express the inverse of the matrix $L^{\pm}(u)$ as follows:
\begin{align}
L^{\pm}(u)^{-1}=&\begin{pmatrix}
1&&&0\\-e_{1}^{\pm}(u)&\ddots\\\vdots&&\ddots\\*&\ldots&-e_{n-1}^{\pm}(u)&1
\end{pmatrix}
\begin{pmatrix}
k_{1}^{\pm}(u)^{-1}&&&0\\&\ddots\\&&\ddots\\0&&&k_{n}^{\pm}(u)^{-1}
\end{pmatrix}\nonumber\\
&\begin{pmatrix}
1&-f_{1}^{\pm}(u)&\ldots&*\\&\ddots\\&&\ddots&-f_{n-1}^{\pm}(u)\\0&&&1
\end{pmatrix}
\end{align}
The defining relations \eqref{e:2.5} and \eqref{e:2.6} imply the following relations:
\begin{align}\label{e:2.13}
L_{1}^{\pm}(v)^{-1}\bar R_{21}(u-v)L_{2}^{\pm}(u)=L_{2}^{\pm}(u)\bar
R_{21}(u-v)L_{1}^{\pm}(v)^{-1}\end{align}
\begin{align}\label{e:2.14}
L_{1}^{-}(v)^{-1}\bar
R_{21}(u_{-}-v_{+})L_{2}^{+}(u)=L_{2}^{+}(u)\bar
R_{21}(u_{+}-v_{-})L_{1}^{-}(v)^{-1}\end{align}
\begin{align}\label{e:2.15}
&\frac{(u_{+}-v_{-})^{2}}{(u_{+}-v_{-})^{2}-h^{2}}\bar R_{21}(u_{+}-v_{-})L_{2}^{-}(u)L_{1}^{+}(v)=\frac{(u_{-}-v_{+})^{2}}{(u_{-}-v_{+})^{2}-h^{2}}L_{1}^{+}(v)\\
&L_{2}^{-}(u)\bar R_{21}(u_{-}-v_{+})\nonumber
\end{align}
\begin{align}\label{e:2.16}
&\frac{(u_{+}-v_{-})^{2}}{(u_{+}-v_{-})^{2}-h^{2}}L_{1}^{+}(v)^{-1}\bar
R_{21}(u_{+}-v_{-})L_{2}^{-}(u)=\frac{(u_{-}-v_{+})^{2}}{(u_{-}-v_{+})^{2}-h^{2}}L_{2}^{-}(u)\\
&\bar R_{21}(u_{-}-v_{+})L_{1}^{+}(v)^{-1}\nonumber
\end{align}
\begin{align}\label{e:2.17}
L_{2}^{\pm}(u)^{-1}L_{1}^{\pm}(v)^{-1}\bar R_{21}(u-v)=\bar
R_{21}(u-v)L_{1}^{\pm}(v)^{-1}L_{2}^{\pm}(u)^{-1}\end{align}
\begin{align}\label{e:2.18}
L_{2}^{+}(u)^{-1}L_{1}^{-}(v)^{-1}\bar R_{21}(u_{-}-v_{+})=\bar
R_{21}(u_{+}-v_{-})L_{1}^{-}(v)^{-1}L_{2}^{+}(u)^{-1}\end{align}

First we consider the case $n=2$. By definition, we have the following formulas:
\begin{align}\label{e:2.19}
L^{\pm}(u)=\begin{pmatrix}
k_{1}^{\pm}(u)&k_{1}^{\pm}(u)e_{1}^{\pm}(u)\\f_{1}^{\pm}(u)k_{1}^{\pm}(u)&k_{2}^{\pm}(u)+f_{1}^{\pm}(u)k_{1}^{\pm}(u)e_{1}^{\pm}(u)
\end{pmatrix}
\end{align}
\begin{align}
L^{\pm}(u)^{-1}=\begin{pmatrix}
k_{1}^{\pm}(u)^{-1}+e_{1}^{\pm}(u)k_{2}^{\pm}(u)^{-1}f_{1}^{\pm}(u)&-e_{1}^{\pm}(u)k_{2}^{\pm}(u)^{-1}\\-k_{2}^{\pm}(u)^{-1}f_{1}^{\pm}(u)&k_{2}^{\pm}(u)^{-1}
\end{pmatrix}
\end{align}
\begin{align}
\bar R_{21}(u)=\bar R(u)=\begin{pmatrix}
\frac{u+h}{u}&0&0&0\\0&1&\frac{h}{u}&0\\0&\frac{h}{u}&1&0\\0&0&0&\frac{u+h}{u}
\end{pmatrix}
\end{align}

\begin{align}
L_{1}^{\pm}(v)^{-1}=\begin{pmatrix}
*&0&-e_{1}^{\pm}(v)k_{2}^{\pm}(v)^{-1}&0\\0&*&0&-e_{1}^{\pm}(v)k_{2}^{\pm}(v)^{-1}\\-k_{2}^{\pm}(v)^{-1}f_{1}^{\pm}(v)&0&k_{2}^{\pm}(v)^{-1}&0\\0&-k_{2}^{\pm}(v)^{-1}f_{1}^{\pm}(v)&0&k_{2}^{\pm}(v)^{-1}
\end{pmatrix}
\end{align}
\begin{align}
L_{2}^{\pm}(u)^{-1}=\begin{pmatrix}
*&-e_{1}^{\pm}(u)k_{2}^{\pm}(u)^{-1}&0&0\\-k_{2}^{\pm}(u)^{-1}f_{1}^{\pm}(u)&k_{2}^{\pm}(u)^{-1}&0&0\\0&0&*&-e_{1}^{\pm}(u)k_{2}^{\pm}(u)^{-1}\\0&0&-k_{2}^{\pm}(u)^{-1}f_{1}^{\pm}(u)&k_{2}^{\pm}(u)^{-1}
\end{pmatrix}
\end{align}

From \eqref{e:2.5} and \eqref{e:2.13}-\eqref{e:2.19}, we get all the relations
among $k_{1}^{\pm}(u),k_{2}^{\pm}(u)$.
\begin{align}
k_{1}^{\pm}(u)k_{1}^{\pm}(v)=k_{1}^{\pm}(v)k_{1}^{\pm}(u)\end{align}
\begin{align}
k_{2}^{\pm}(u)k_{2}^{\pm}(v)=k_{2}^{\pm}(v)k_{2}^{\pm}(u)\end{align}
\begin{align}
\frac{u_{-}-v_{+}+h}{u_{-}-v_{+}}k_{1}^{+}(u)k_{1}^{-}(v)=\frac{u_{+}-v_{-}+h}{u_{+}-v_{-}}k_{1}^{-}(v)k_{1}^{+}(u)\end{align}
\begin{align}
\frac{u_{-}-v_{+}+h}{u_{-}-v_{+}}k_{2}^{+}(u)k_{2}^{-}(v)=\frac{u_{+}-v_{-}+h}{u_{+}-v_{-}}k_{2}^{-}(v)k_{2}^{+}(u)\end{align}
\begin{align}
k_{1}^{\pm}(u)k_{2}^{\pm}(v)=k_{2}^{\pm}(v)k_{1}^{\pm}(u)\end{align}
\begin{align}
k_{2}^{-}(v)^{-1}k_{1}^{+}(u)=k_{1}^{+}(u)k_{2}^{-}(v)^{-1}\end{align}
\begin{align}
\frac{(u_{+}-v_{-})^{2}}{(u_{+}-v_{-})^{2}-h^{2}}k_{2}^{+}(v)^{-1}k_{1}^{-}(u)=\frac{(u_{-}-v_{+})^{2}}{(u_{-}-v_{+})^{2}-h^{2}}k_{1}^{-}(u)k_{2}^{+}(v)^{-1}\end{align}

Next we write down the relations between $k_{1}^{\pm}(u)$ and
$e_{1}^{\pm}(u)$ or $f_{1}^{\pm}(u)$.
\begin{align}
\frac{u-v+h}{u-v}k_{1}^{\pm}(u)k_{1}^{\pm}(v)e_{1}^{\pm}(v)=k_{1}^{\pm}(v)e_{1}^{\pm}(v)k_{1}^{\pm}(u)+\frac{h}{u-v}k_{1}^{\pm}(v)k_{1}^{\pm}(u)e_{1}^{\pm}(u)\end{align}
\begin{align}
\frac{u-v+h}{u-v}f_{1}^{\pm}(v)k_{1}^{\pm}(v)k_{1}^{\pm}(u)=\frac{h}{u-v}f_{1}^{\pm}(u)k_{1}^{\pm}(u)k_{1}^{\pm}(v)+k_{1}^{\pm}(u)f_{1}^{\pm}(v)k_{1}^{\pm}(v)\end{align}
\begin{align}
\frac{u_{-}-v_{+}+h}{u_{-}-v_{+}}k_{1}^{+}(u)k_{1}^{-}(v)e_{1}^{-}(v)=k_{1}^{-}(v)e_{1}^{-}(v)k_{1}^{+}(u)+\frac{h}{u_{+}-v_{-}}k_{1}^{-}(v)k_{1}^{+}(u)e_{1}^{+}(u)\end{align}
\begin{align}
\frac{u_{+}-v_{-}+h}{u_{+}-v_{-}}f_{1}^{-}(v)k_{1}^{-}(v)k_{1}^{+}(u)=\frac{h}{u_{-}-v_{+}}f_{1}^{+}(u)k_{1}^{+}(u)k_{1}^{-}(v)+k_{1}^{+}(u)f_{1}^{-}(v)k_{1}^{-}(v)\end{align}

Thus
\begin{align}\label{e:2.35}
k_{1}^{\pm}(u)^{-1}e_{1}^{\pm}(v)k_{1}^{\pm}(u)=\frac{u-v+h}{u-v}e_{1}^{\pm}(v)-\frac{h}{u-v}e_{1}^{\pm}(u)\end{align}
\begin{align}\label{e:2.36}
k_{1}^{\pm}(u)f_{1}^{\pm}(v)k_{1}^{\pm}(u)^{-1}=\frac{u-v+h}{u-v}f_{1}^{\pm}(v)-\frac{h}{u-v}f_{1}^{\pm}(u)\end{align}
\begin{align}\label{e:2.37}
k_{1}^{+}(u)^{-1}e_{1}^{-}(v)k_{1}^{+}(u)=\frac{u_{+}-v_{-}+h}{u_{+}-v_{-}}e_{1}^{-}(v)-\frac{h}{u_{+}-v_{-}}e_{1}^{+}(u)\end{align}
\begin{align}
k_{1}^{+}(u)f_{1}^{-}(v)k_{1}^{+}(u)^{-1}=\frac{u_{-}-v_{+}+h}{u_{-}-v_{+}}f_{1}^{-}(v)-\frac{h}{u_{-}-v_{+}}f_{1}^{+}(u)\end{align}

Similarly we have
\begin{align}
k_{1}^{-}(u)^{-1}e_{1}^{+}(v)k_{1}^{-}(u)=\frac{u_{-}-v_{+}+h}{u_{-}-v_{+}}e_{1}^{+}(v)-\frac{h}{u_{-}-v_{+}}e_{1}^{-}(u)\end{align}
\begin{align}
k_{1}^{-}(u)f_{1}^{+}(v)k_{1}^{-}(u)^{-1}=\frac{u_{+}-v_{-}+h}{u_{+}-v_{-}}f_{1}^{+}(v)-\frac{h}{u_{+}-v_{-}}f_{1}^{-}(u)\end{align}

Then we have
\begin{align}
k_{1}^{\pm}(u)^{-1}e_{1}^{\mp}(v)k_{1}^{\pm}(u)&=\frac{u_{\pm}-v_{\mp}+h}{u_{\pm}-v_{\mp}}e_{1}^{\mp}(v)-\frac{h}{u_{\pm}-v_{\mp}}e_{1}^{\pm}(u), \\
k_{1}^{\pm}(u)f_{1}^{\mp}(v)k_{1}^{\pm}(u)^{-1}&=\frac{u_{\mp}-v_{\pm}+h}{u_{\mp}-v_{\pm}}f_{1}^{\mp}(v)-\frac{h}{u_{\mp}-v_{\pm}}f_{1}^{\pm}(u).
\end{align}

Since
$X_{1}^{+}(u)=e_{1}^{+}(u_{-})-e_{1}^{-}(u_{+}),X_{1}^{-}(u)=f_{1}^{+}(u_{+})-f_{1}^{-}(u_{-})$,
so
\begin{align}
k_{1}^{\pm}(u)^{-1}X_{1}^{+}(v)k_{1}^{\pm}(u)&=\frac{u_{\pm}-v+h}{u_{\pm}-v}X_{1}^{+}(v),\\
k_{1}^{\pm}(u)X_{1}^{-}(v)k_{1}^{\pm}(u)^{-1}&=\frac{u_{\mp}-v+h}{u_{\mp}-v}X_{1}^{-}(v).
\end{align}

Then we obtain the relations among
$f_{1}^{\pm}(u),f_{2}^{\pm}(v)$ and $e_{1}^{\pm}(u),e_{2}^{\pm}(v)$.
\begin{align}\label{e:2.45}
k_{1}^{\pm}(u)e_{1}^{\pm}(u)k_{1}^{\pm}(v)e_{1}^{\pm}(v)&=k_{1}^{\pm}(v)e_{1}^{\pm}(v)k_{1}^{\pm}(u)e_{1}^{\pm}(u), \\ \label{e:2.46}
f_{1}^{\pm}(v)k_{1}^{\pm}(v)f_{1}^{\pm}(u)k_{1}^{\pm}(u)&=f_{1}^{\pm}(u)k_{1}^{\pm}(u)f_{1}^{\pm}(v)k_{1}^{\pm}(v), \end{align}
\begin{align}\label{e:2.47}
\frac{u_{-}-v_{+}+h}{u_{-}-v_{+}}k_{1}^{+}(u)e_{1}^{+}(u)k_{1}^{-}(v)e_{1}^{-}(v)=\frac{u_{+}
-v_{-}+h}{u_{+}-v_{-}}k_{1}^{-}(v)e_{1}^{-}(v)k_{1}^{+}(u)e_{1}^{+}(u),\end{align}
\begin{align}
\frac{u_{-}-v_{+}+h}{u_{-}-v_{+}}f_{1}^{+}(u)k_{1}^{+}(u)f_{1}^{-}(v)k_{1}^{-}(v)=\frac{u_{+}-v_{-}+h}{u_{+}
-v_{-}}f_{1}^{-}(v)k_{1}^{-}(v)f_{1}^{+}(u)k_{1}^{+}(u). \end{align}

From \eqref{e:2.35} and \eqref{e:2.45}, we get
\begin{multline}
\frac{u-v+h}{u-v}e_{1}^{\pm}(v)e_{1}^{\pm}(u)-\frac{h}{u-v}e_{1}^{\pm}(u)e_{1}^{\pm}(u)\\=\frac{v-u+h}{v-u}e_{1}^{\pm}(u)e_{1}^{\pm}(v)-\frac{h}{v-u}e_{1}^{\pm}(v)e_{1}^{\pm}(v).\end{multline}
From \eqref{e:2.37} and \eqref{e:2.47}, we get
\begin{multline}
\frac{u_{+}-v_{-}+h}{u_{+}-v_{-}}e_{1}^{-}(v)e_{1}^{+}(u)-\frac{h}{u_{+}-v_{-}}e_{1}^{+}(u)e_{1}^{+}(u)\\
=\frac{v_{-}-u_{+}+h}{v_{-}-u_{+}}e_{1}^{+}(u)e_{1}^{-}(v)-\frac{h}{v_{-}-u_{+}}e_{1}^{-}(v)e_{1}^{-}(v).\end{multline}
Similarly, we can prove
\begin{multline}
\frac{u_{-}-v_{+}+h}{u_{-}-v_{+}}e_{1}^{+}(v)e_{1}^{-}(u)-\frac{h}{u_{-}-v_{+}}e_{1}^{-}(u)e_{1}^{-}(u)\\
=\frac{v_{+}-u_{-}+h}{v_{+}-u_{-}}e_{1}^{-}(u)e_{1}^{+}(v)-\frac{h}{v_{+}-u_{-}}e_{1}^{+}(v)e_{1}^{+}(v).\end{multline}

Therefore,
\begin{multline}
\frac{u_{\pm}-v_{\mp}+h}{u_{\pm}-v_{\mp}}e_{1}^{\mp}(v)e_{1}^{\pm}(u)-\frac{h}{u_{\pm}-v_{\mp}}e_{1}^{\pm}(u)e_{1}^{\pm}(u)\\
=\frac{v_{\mp}-u_{\pm}+h}{v_{\mp}-u_{\pm}}e_{1}^{\pm}(u)e_{1}^{\mp}(v)-\frac{h}{v_{\mp}-u_{\pm}}e_{1}^{\mp}(v)e_{1}^{\mp}(v).\end{multline}

From \eqref{e:2.36} and \eqref{e:2.46}, we get
\begin{multline}
\frac{u-v+h}{u-v}f_{1}^{\pm}(u)f_{1}^{\pm}(v)-\frac{h}{u-v}f_{1}^{\pm}(u)f_{1}^{\pm}(u)\\
=\frac{v-u+h}{v-u}f_{1}^{\pm}(v)f_{1}^{\pm}(u)-\frac{h}{v-u}f_{1}^{\pm}(v)f_{1}^{\pm}(v).\end{multline}

From (2.38) and (2.48), we get
\begin{multline}
\frac{u_{-}-v_{+}+h}{u_{-}-v_{+}}f_{1}^{+}(u)f_{1}^{-}(v)-\frac{h}{u_{-}-v_{+}}f_{1}^{+}(u)f_{1}^{+}(u)\\
=\frac{v_{+}-u_{-}+h}{v_{+}-u_{-}}f_{1}^{-}(v)f_{1}^{+}(u)-\frac{h}{v_{+}-u_{-}}f_{1}^{-}(v)f_{1}^{-}(v).\end{multline}

Similarly, we can prove
\begin{multline}
\frac{u_{+}-v_{-}+h}{u_{+}-v_{-}}f_{1}^{-}(u)f_{1}^{+}(v)-\frac{h}{u_{+}-v_{-}}f_{1}^{-}(u)f_{1}^{-}(u)\\
=\frac{v_{-}-u_{+}+h}{v_{-}-u_{+}}f_{1}^{+}(v)f_{1}^{-}(u)-\frac{h}{v_{-}-u_{+}}f_{1}^{+}(v)f_{1}^{+}(v).\end{multline}

Hence,
\begin{multline}
\frac{u_{\mp}-v_{\pm}+h}{u_{\mp}-v_{\pm}}f_{1}^{\pm}(u)f_{1}^{\mp}(v)-\frac{h}{u_{\mp}-v_{\pm}}f_{1}^{\pm}(u)f_{1}^{\pm}(u)\\
=\frac{v_{\pm}-u_{\mp}+h}{v_{\pm}-u_{\mp}}f_{1}^{\mp}(v)f_{1}^{\pm}(u)-\frac{h}{v_{\pm}-u_{\mp}}f_{1}^{\mp}(v)f_{1}^{\mp}(v).\end{multline}

From (2.49) and (2.52), we get
\begin{align}
(u-v-h)X_{1}^{+}(u)X_{1}^{+}(v)=(u-v+h)X_{1}^{+}(v)X_{1}^{+}(u).\end{align}

From (2.53) and (2.56), we get
\begin{align}
(u-v+h)X_{1}^{-}(u)X_{1}^{-}(v)=(u-v-h)X_{1}^{-}(v)X_{1}^{-}(u).\end{align}

The relations between $f_{1}^{\pm}(u),e_{1}^{\pm}(u)$ and
$k_{2}^{\pm}(u)$ are:
\begin{multline}
-\frac{u-v+h}{u-v}e_{1}^{\pm}(u)k_{2}^{\pm}(u)^{-1}k_{2}^{\pm}(v)^{-1}\\
=-k_{2}^{\pm}(v)^{-1}e_{1}^{\pm}(u)k_{2}^{\pm}(u)^{-1}-\frac{h}{u-v}e_{1}^{\pm}(v)k_{2}^{\pm}(v)^{-1}k_{2}^{\pm}(u)^{-1},\end{multline}
\begin{multline}
-\frac{u-v+h}{u-v}k_{2}^{\pm}(v)^{-1}k_{2}^{\pm}(u)^{-1}f_{1}^{\pm}(u)\\
=-k_{2}^{\pm}(u)^{-1}f_{1}^{\pm}(u)k_{2}^{\pm}(v)^{-1}-\frac{h}{u-v}k_{2}^{\pm}(u)^{-1}k_{2}^{\pm}(v)^{-1}f_{1}^{\pm}(v),\end{multline}
\begin{multline}
-\frac{u_{-}-v_{+}+h}{u_{-}-v_{+}}e_{1}^{+}(u)k_{2}^{+}(u)^{-1}k_{2}^{-}(v)^{-1}\\
=-k_{2}^{-}(v)^{-1}e_{1}^{+}(u)k_{2}^{+}(u)^{-1}-\frac{h}{u_{+}-v_{-}}e_{1}^{-}(v)k_{2}^{-}(v)^{-1}k_{2}^{+}(u)^{-1},\end{multline}
\begin{multline}
-\frac{u_{+}-v_{-}+h}{u_{+}-v_{-}}k_{2}^{-}(v)^{-1}k_{2}^{+}(u)^{-1}f_{1}^{+}(u)\\
=-k_{2}^{+}(u)^{-1}f_{1}^{+}(u)k_{2}^{-}(v)^{-1}-\frac{h}{u_{-}-v_{+}}k_{2}^{+}(u)^{-1}k_{2}^{-}(v)^{-1}f_{1}^{-}(v).\end{multline}

Then
\begin{align}\label{e:2.63}
k_{2}^{\pm}(v)^{-1}e_{1}^{\pm}(u)k_{2}^{\pm}(v)&=\frac{u-v+h}{u-v}e_{1}^{\pm}(u)-\frac{h}{u-v}e_{1}^{\pm}(v),\\ \label{e:2.64}
k_{2}^{\pm}(v)f_{1}^{\pm}(u)k_{2}^{\pm}(v)^{-1}&=\frac{u-v+h}{u-v}f_{1}^{\pm}(u)-\frac{h}{u-v}f_{1}^{\pm}(v),\\
k_{2}^{-}(v)^{-1}e_{1}^{+}(u)k_{2}^{-}(v)&=\frac{u_{+}-v_{-}+h}{u_{+}-v_{-}}e_{1}^{+}(u)-\frac{h}{u_{+}-v_{-}}e_{1}^{-}(v),\\
k_{2}^{-}(v)f_{1}^{+}(u)k_{2}^{-}(v)^{-1}&=\frac{u_{-}-v_{+}+h}{u_{-}-v_{+}}f_{1}^{+}(u)-\frac{h}{u_{-}-v_{+}}f_{1}^{-}(v).\end{align}

Similarly, we can prove
\begin{align}
k_{2}^{+}(v)^{-1}e_{1}^{-}(u)k_{2}^{+}(v)&=\frac{u_{-}-v_{+}+h}{u_{-}-v_{+}}e_{1}^{-}(u)-\frac{h}{u_{-}-v_{+}}e_{1}^{+}(v), \\
k_{2}^{+}(v)f_{1}^{-}(u)k_{2}^{+}(v)^{-1}&=\frac{u_{+}-v_{-}+h}{u_{+}-v_{-}}f_{1}^{-}(u)-\frac{h}{u_{+}-v_{-}}f_{1}^{+}(v). \end{align}

So we have
\begin{align}\label{e:2.69}
k_{2}^{\mp}(v)^{-1}e_{1}^{\pm}(u)k_{2}^{\mp}(v)&=\frac{u_{\pm}-v_{\mp}+h}{u_{\pm}-v_{\mp}}e_{1}^{\pm}(u)-\frac{h}{u_{\pm}-v_{\mp}}e_{1}^{\mp}(v), \\ \label{e:2.70}
k_{2}^{\mp}(v)f_{1}^{\pm}(u)k_{2}^{\mp}(v)^{-1}&=\frac{u_{\mp}-v_{\pm}+h}{u_{\mp}-v_{\pm}}f_{1}^{\pm}(u)-\frac{h}{u_{\mp}-v_{\pm}}f_{1}^{\mp}(v).\end{align}

From \eqref{e:2.63} and \eqref{e:2.69}, we get
\begin{align}
k_{2}^{\pm}(u)^{-1}X_{1}^{+}(v)k_{2}^{\pm}(u)=\frac{u_{\pm}-v+h}{u_{\pm}-v}X_{1}^{+}(v).\end{align}
From \eqref{e:2.64} and \eqref{e:2.70}, we get
\begin{align}
k_{2}^{\pm}(u)X_{1}^{-}(v)k_{2}^{\pm}(u)^{-1}=\frac{u_{\mp}-v+h}{u_{\mp}-v}X_{1}^{-}(v).\end{align}

The relations between
$k_{1}^{\pm}(u),k_{2}^{\pm}(u),e_{1}^{\pm}(u),f_{1}^{\pm}(u)$ are:
\begin{multline}
\frac{h}{u-v}(k_{2}^{\pm}(v)+f_{1}^{\pm}(v)k_{1}^{\pm}(v)e_{1}^{\pm}(v))k_{1}^{\pm}(u)+f_{1}^{\pm}(v)k_{1}^{\pm}(v)k_{1}^{\pm}(u)e_{1}^{\pm}(u)\\
=\frac{h}{u-v}(k_{2}^{\pm}(u)+f_{1}^{\pm}(u)k_{1}^{\pm}(u)e_{1}^{\pm}(u))k_{1}^{\pm}(v)+k_{1}^{\pm}(u)e_{1}^{\pm}(u)f_{1}^{\pm}(v)k_{1}^{\pm}(v),\end{multline}
\begin{multline}
\frac{h}{u_{+}-v_{-}}(k_{2}^{-}(v)+f_{1}^{-}(v)k_{1}^{-}(v)e_{1}^{-}(v))k_{1}^{+}(u)+f_{1}^{-}(v)k_{1}^{-}(v)k_{1}^{+}(u)e_{1}^{+}(u)\\
=\frac{h}{u_{-}-v_{+}}(k_{2}^{+}(u)+f_{1}^{+}(u)k_{1}^{+}(u)e_{1}^{+}(u))k_{1}^{-}(v)+k_{1}^{+}(u)e_{1}^{+}(u)f_{1}^{-}(v)k_{1}^{-}(v).\end{multline}
Therefore,\begin{align*}
&\frac{h}{u-v}(k_{2}^{\pm}(v)k_{1}^{\pm}(u)-k_{2}^{\pm}(u)k_{1}^{\pm}(v))\\
=&\frac{h}{u-v}f_{1}^{\pm}(u)k_{1}^{\pm}(u)e_{1}^{\pm}(u)k_{1}^{\pm}(v)+
k_{1}^{\pm}(u)e_{1}^{\pm}(u)f_{1}^{\pm}(v)k_{1}^{\pm}(v)\\
&-f_{1}^{\pm}(v)(\frac{h}{u-v}k_{1}^{\pm}(v)e_{1}^{\pm}(v)k_{1}^{\pm}(u)+k_{1}^{\pm}(v)k_{1}^{\pm}(u)e_{1}^{\pm}(u))\\
=&\frac{h}{u-v}f_{1}^{\pm}(u)k_{1}^{\pm}(u)e_{1}^{\pm}(u)k_{1}^{\pm}(v)+k_{1}^{\pm}(u)e_{1}^{\pm}(u)f_{1}^{\pm}(v)k_{1}^{\pm}(v)-\frac{u-v+h}{u-v}\\
&f_{1}^{\pm}(v)k_{1}^{\pm}(u)e_{1}^{\pm}(u)k_{1}^{\pm}(v)\\
=&(\frac{h}{u-v}f_{1}^{\pm}(u)k_{1}^{\pm}(u)-\frac{u-v+h}{u-v}f_{1}^{\pm}(v)k_{1}^{\pm}(u))e_{1}^{\pm}(u)k_{1}^{\pm}(v)\\
&+k_{1}^{\pm}(u)e_{1}^{\pm}(u)f_{1}^{\pm}(v)k_{1}^{\pm}(v)\\
=&-k_{1}^{\pm}(u)f_{1}^{\pm}(v)e_{1}^{\pm}(u)k_{1}^{\pm}(v)+k_{1}^{\pm}(u)e_{1}^{\pm}(u)f_{1}^{\pm}(v)k_{1}^{\pm}(v)\\
=&k_{1}^{\pm}(u)[e_{1}^{\pm}(u),f_{1}^{\pm}(v)]k_{1}^{\pm}(v).
\end{align*}
So
\begin{align}
[e_{1}^{\pm}(u),f_{1}^{\pm}(v)]=\frac{h}{u-v}(k_{2}^{\pm}(v)k_{1}^{\pm}(v)^{-1}-k_{2}^{\pm}(u)k_{1}^{\pm}(u)^{-1}).\end{align}
Similarly, we can prove
\begin{align}
[e_{1}^{\pm}(u),f_{1}^{\mp}(v)]=\frac{h}{u_{\pm}-v_{\mp}}k_{2}^{\mp}(v)k_{1}^{\mp}(v)^{-1}-\frac{h}{u_{\mp}-v_{\pm}}k_{2}^{\pm}(u)k_{1}^{\pm}(u)^{-1}.\end{align}
Here the denominators are power series in $\frac{v}{u}$ and
$\frac{u}{v}$ respectively. Since
$X_{1}^{+}(u)=e_{1}^{+}(u_{-})-e_{1}^{-}(u_{+}),X_{1}^{-}(v)=f_{1}^{+}(v_{+})-f_{1}^{-}(v_{-})$,
\begin{multline}
[X_{1}^{+}(u),X_{1}^{-}(v)]=h\{\delta(u_{-}-v_{+})k_{2}^{+}(u_{-})k_{1}^{+}(u_{-})^{-1}\\
+\delta(u_{+}-v_{-})k_{2}^{-}(v_{-})k_{1}^{-}(v_{-})^{-1}\}.\end{multline}

Now we consider the case $n=3$. In this section, we will denote $\bar R(u)$
by $\bar R_{n}(u)$ for referring dimension $n$. Let us restrict
\eqref{e:2.5} and \eqref{e:2.6} to $e_{ij}\otimes e_{kl},i,j,k,l\leq 2$, then we get
\begin{align}
\bar R_{2}(u-v)J_1^{\pm}(u)J_2^{\pm}(v)&=J_2^{\pm}(v)J_1^{\pm}(u)\bar
R_{2}(u-v), \\
\bar R_{2}(u_{-}-v_{+})J_1^{+}(u)J_2^{-}(v)&=J_2^{-}(v)J_1^{+}(u)\bar
R_{2}(u_{+}-v_{-}), \end{align}
where we have denoted
\begin{align}
J^{\pm}(u)=\begin{pmatrix} 1&0\\f_{1}^{\pm}(u)&1
\end{pmatrix}
\begin{pmatrix} k_{1}^{\pm}(u)&0\\0&k_{2}^{\pm}(u)
\end{pmatrix}
\begin{pmatrix} 1&e_{1}^{\pm}(u)\\0&1
\end{pmatrix}.
\end{align}
Thus we derive the same relations as in the case $n=2$. Similarly consider
\eqref{e:2.5} and \eqref{e:2.6} and restrict them to $e_{ij}\otimes e_{kl},2\leq
i,j,k,l\leq 3$, then we have
\begin{align}
\bar
R_{2}(u-v)\tilde{J}_1^{\pm}(u)\tilde{J}_2^{\pm}(v)&=\tilde{J}_2^{\pm}(v)\tilde{J}_1^{\pm}(u)\bar
R_{2}(u-v), \\
\bar
R_{2}(u_{-}-v_{+})\tilde{J}_1^{+}(u)\tilde{J}_2^{-}(v)&=\tilde{J}_2^{-}(v)\tilde{J}_1^{+}(u)\bar
R_{2}(u_{+}-v_{-})\end{align} where we have denoted
\begin{align}
\tilde{J}^{\pm}(u)=\begin{pmatrix} 1&0\\f_{2}^{\pm}(u)&1
\end{pmatrix}
\begin{pmatrix} k_{2}^{\pm}(u)&0\\0&k_{3}^{\pm}(u)
\end{pmatrix}
\begin{pmatrix} 1&e_{2}^{\pm}(u)\\0&1
\end{pmatrix}
\end{align}

It is also the same as in $n=2$ case. Thus we only need to verify the
relations between $k_{3}^{\pm}(u),f_{1}^{\pm}(u),e_{1}^{\pm}(u)$ and
$k_{3}^{\pm}(u),f_{2}^{\pm}(u),e_{2}^{\pm}(u)$. The other relations
can be deduced from the results for $n=2$ case.

By the Gauss decomposition of $L^{\pm}(u)$, we have
\begin{align}
L^{\pm}(u)=&\begin{pmatrix}
1&0&0\\f_{1}^{\pm}(u)&1&0\\f_{3,1}^{\pm}(u)&f_{2}^{\pm}(u)&1
\end{pmatrix}
\begin{pmatrix}
k_{1}^{\pm}(u)&0&0\\0&k_{2}^{\pm}(u)&0\\0&0&k_{3}^{\pm}(u)
\end{pmatrix}
\begin{pmatrix}
1&e_{1}^{\pm}(u)&e_{1,3}^{\pm}(u)\\0&1&e_{2}^{\pm}(u)\\0&0&1
\end{pmatrix}\\
=&\begin{pmatrix}
k_{1}^{\pm}(u)&k_{1}^{\pm}(u)e_{1}^{\pm}(u)&k_{1}^{\pm}(u)e_{1,3}^{\pm}(u)\\f_{1}^{\pm}(u)k_{1}^{\pm}(u)&*&*\\f_{3,1}^{\pm}(u)k_{1}^{\pm}(u)&*&*
\end{pmatrix}
\end{align}

Let
$x^{\pm}=k_{3}^{\pm}(v)^{-1}(-f_{3,1}^{\pm}(v)+f_{2}^{\pm}(v)f_{1}^{\pm}(v)),y^{\pm}=(-e_{1,3}^{\pm}(v)+e_{1}^{\pm}(v)e_{2}^{\pm}(v))$\\
$k_{3}^{\pm}(v)^{-1}$, then
\begin{align}\label{e:2.86}
L_{1}^{\pm}(v)^{-1}=\begin{pmatrix}
*&*&y^{\pm}I\\*&*&-e_{2}^{\pm}(v)k_{3}^{\pm}(v)^{-1}I\\x^{\pm}I&-k_{3}^{\pm}(v)^{-1}f_{2}^{\pm}(v)I&k_{3}^{\pm}(v)^{-1}I
\end{pmatrix}
\end{align}

Now we start to check the relations between $f_{1}^{\pm}(u)$ and
$f_{2}^{\pm}(u)$, and the relations between $e_{1}^{\pm}(u)$ and
$e_{2}^{\pm}(u)$.

From \eqref{e:2.14}, \eqref{e:2.15}, \eqref{e:2.16} and  \eqref{e:2.86}, we have
\begin{multline}\label{e:2.87}
\frac{h}{u-v}x^{\pm}k_{1}^{\pm}(u)-\frac{u-v+h}{u-v}k_{3}^{\pm}(v)^{-1}f_{2}^{\pm}(v)f_{1}^{\pm}(u)k_{1}^{\pm}(u)\\
+\frac{h}{u-v}k_{3}^{\pm}(v)^{-1}f_{3,1}^{\pm}(u)k_{1}^{\pm}(u)=
-f_{1}^{\pm}(u)k_{1}^{\pm}(u)k_{3}^{\pm}(v)^{-1}f_{2}^{\pm}(v)\end{multline}
\begin{multline}\label{e:2.88}
\frac{h}{u_{-}-v_{+}}x^{-}k_{1}^{+}(u)-\frac{u_{-}-v_{+}+h}{u_{-}-v_{+}}k_{3}^{-}(v)^{-1}f_{2}^{-}(v)f_{1}^{+}(u)k_{1}^{+}(u)\\
+\frac{h}{u_{-}-v_{+}}k_{3}^{-}(v)^{-1}f_{3,1}^{+}(u)k_{1}^{+}(u)=
-f_{1}^{+}(u)k_{1}^{+}(u)k_{3}^{-}(v)^{-1}f_{2}^{-}(v)\end{multline}

We multiply  \eqref{e:2.87} by $k_{3}^{\pm}(v)$ on the left side, and
$k_{1}^{\pm}(u)^{-1}$ on the right side, and  \eqref{e:2.88} by
$k_{3}^{-}(v)$ on the left, and by $k_{1}^{+}(u)^{-1}$ on the
right.\\Since
$k_{3}^{\pm}(v)^{-1}k_{1}^{\pm}(u)=k_{1}^{\pm}(u)k_{3}^{\pm}(v)^{-1}$
and
$k_{3}^{-}(v)^{-1}k_{1}^{+}(u)=k_{1}^{+}(u)k_{3}^{-}(v)^{-1}$,
then we have
\begin{multline}\label{e:2.89}
(u-v)f_{1}^{\pm}(u)f_{2}^{\pm}(v)=(u-v+h)f_{2}^{\pm}(v)f_{1}^{\pm}(u)\\
+h(f_{3,1}^{\pm}(v)-f_{3,1}^{\pm}(u)-f_{2}^{\pm}(v)f_{1}^{\pm}(v)),\end{multline}
\begin{multline}\label{e:2.90}
(u_{-}-v_{+})f_{1}^{+}(u)f_{2}^{-}(v)=(u_{-}-v_{+}+h)f_{2}^{-}(v)f_{1}^{+}(u)\\
+h(f_{3,1}^{-}(v)-f_{3,1}^{+}(u)-f_{2}^{-}(v)f_{1}^{-}(v)).\end{multline}
Similarly, we can prove
\begin{multline}
(u_{+}-v_{-})f_{1}^{-}(u)f_{2}^{+}(v)=(u_{+}-v_{-}+h)f_{2}^{+}(v)f_{1}^{-}(u)\\
+h(f_{3,1}^{+}(v)-f_{3,1}^{-}(u)-f_{2}^{+}(v)f_{1}^{+}(v)).\end{multline}
So we have
\begin{multline}
(u_{\mp}-v_{\pm})f_{1}^{\pm}(u)f_{2}^{\mp}(v)=(u_{\mp}-v_{\pm}+h)f_{2}^{\mp}(v)f_{1}^{\pm}(u)\\
+h(f_{3,1}^{\mp}(v)-f_{3,1}^{\pm}(u)-f_{2}^{\mp}(v)f_{1}^{\mp}(v)).\end{multline}

From \eqref{e:2.89}, we get
\begin{multline}\label{e:2.93}
(u-v)f_{1}^{\pm}(u_{\pm})f_{2}^{\pm}(v_{\pm})=(u-v+h)f_{2}^{\pm}(v_{\pm})f_{1}^{\pm}(u_{\pm})\\
+h(f_{3,1}^{\pm}(v_{\pm})-f_{3,1}^{\pm}(u_{\pm})-f_{2}^{\pm}(v_{\pm})f_{1}^{\pm}(v_{\pm})).\end{multline}

From \eqref{e:2.90}, we get
\begin{multline}\label{e:2.94}
(u-v)f_{1}^{\pm}(u_{\pm})f_{2}^{\mp}(v_{\mp})=(u-v+h)f_{2}^{\mp}(v_{\mp})f_{1}^{\pm}(u_{\pm})\\
+h(f_{3,1}^{\mp}(v_{\mp})-f_{3,1}^{\pm}(u_{\pm})-f_{2}^{\mp}(v_{\mp})f_{1}^{\mp}(v_{\mp})).\end{multline}
From \eqref{e:2.93}-\eqref{e:2.94}, we get
\begin{align}\label{e:2.95}
(u-v)X_{1}^{-}(u)X_{2}^{-}(v)=X_{2}^{-}(v)X_{1}^{-}(u)(u-v+h).\end{align}

As for $e_{1}^{\pm}(u),e_{2}^{\pm}(u)$, we have
\begin{multline}
-e_{2}^{\pm}(v)k_{3}^{\pm}(v)^{-1}k_{1}^{\pm}(u)e_{1}^{\pm}(u)=\frac{h}{u-v}k_{1}^{\pm}(u)y^{\pm}\\
-\frac{u-v}{u-v+h}k_{1}^{\pm}(u)e_{1}^{\pm}(u)e_{2}^{\pm}(v)k_{3}^{\pm}(v)^{-1}+
\frac{h}{u-v}k_{1}^{\pm}(u)e_{1,3}^{\pm}(u)k_{3}^{\pm}(v)^{-1}\end{multline}
\begin{multline}
-e_{2}^{-}(v)k_{3}^{-}(v)^{-1}k_{1}^{+}(u)e_{1}^{+}(u)=\frac{h}{u_{+}-v_{-}}k_{1}^{+}(u)y^{-}\\
-\frac{u_{+}-v_{-}+h}{u_{+}-v_{-}}k_{1}^{+}(u)e_{1}^{+}(u)e_{2}^{-}(v)k_{3}^{-}(v)^{-1}+
\frac{h}{u_{+}-v_{-}}k_{1}^{+}(u)e_{1,3}^{+}(u)k_{3}^{-}(v)^{-1}\end{multline}
Then we get
\begin{multline}\label{e:2.98}
(u-v)e_{2}^{\pm}(v_{\mp})e_{1}^{\pm}(u_{\mp})=(u-v+h)e_{1}^{\pm}(u_{\mp})e_{2}^{\pm}(v_{\mp})\\
+h(e_{1,3}^{\pm}(v_{\mp})-e_{1,3}^{\pm}(u_{\mp})-e_{1}^{\pm}(v_{\mp})e_{2}^{\pm}(v_{\mp}))\end{multline}
\begin{multline}\label{e:2.99}
(u-v)e_{2}^{\mp}(v_{\pm})e_{1}^{\pm}(u_{\mp})=(u-v+h)e_{1}^{\pm}(u_{\mp})e_{2}^{\mp}(v_{\pm})\\
+h(e_{1,3}^{\mp}(v_{\pm})-e_{1,3}^{\pm}(u_{\mp})-e_{1}^{\mp}(v_{\pm})e_{2}^{\mp}(v_{\pm}))\end{multline}

From \eqref{e:2.98}-\eqref{e:2.99}, we get
\begin{align}\label{e:2.100}
(u-v)X_{2}^{+}(v)X_{1}^{+}(u)=X_{1}^{+}(u)X_{2}^{+}(v)(u-v+h).\end{align}

From \eqref{e:2.95} and
$(u-v+h)X_{1}^{-}(u)X_{1}^{-}(v)=X_{1}^{-}(v)X_{1}^{-}(u)(u-v-h)$,\\
we get
\begin{multline}
X_{1}^{-}(u_{1})X_{1}^{-}(u_{2})X_{2}^{-}(v)-2X_{1}^{-}(u_{1})X_{2}^{-}(v)X_{1}^{-}(u_{2})\\
+X_{2}^{-}(v)X_{1}^{-}(u_{1})X_{1}^{-}(u_{2})+\{u_{1}\leftrightarrow
u_{2}\}=0.\end{multline}

From \label{e:2.100} and
$(u-v-h)X_{1}^{+}(u)X_{1}^{+}(v)=X_{1}^{+}(v)X_{1}^{+}(u)(u-v+h)$,\\
we get
\begin{multline}
X_{1}^{+}(u_{1})X_{1}^{+}(u_{2})X_{2}^{+}(v)-2X_{1}^{+}(u_{1})X_{2}^{+}(v)X_{1}^{+}(u_{2})\\
+X_{2}^{+}(v)X_{1}^{+}(u_{1})X_{1}^{+}(u_{2})+\{u_{1}\leftrightarrow
u_{2}\}=0.\end{multline}

Thus we have proved all the 
relations for $n=3$. Now we
proceed to the proof for general $n$. First of all, we
restrict \eqref{e:2.5} and \eqref{e:2.6} to $e_{ij}\otimes e_{kl},1\leq
i,j,k,l\leq n-1$, then we get
\begin{align}
\bar
R_{n-1}(u-v)J_1^{\pm}(u)J_2^{\pm}(v)=J_2^{\pm}(v)J_1^{\pm}(u)\bar
R_{n-1}(u-v)\end{align}
\begin{align}
\bar
R_{n-1}(u_{-}-v_{+})J_1^{+}(u)J_2^{-}(v)=J_2^{-}(v)J_1^{+}(u)\bar
R_{n-1}(u_{+}-v_{-})\end{align}
\begin{align}
J^{\pm}(u)=&\begin{pmatrix}
1&&&0\\e_{1}^{\pm}(u)&\ddots\\\vdots&&\ddots\\*&\ldots&e_{n-2}^{\pm}(u)&1
\end{pmatrix}
\begin{pmatrix}
k_{1}^{\pm}(u)&&&0\\&\ddots\\&&\ddots\\0&&&k_{n-1}^{\pm}(u)
\end{pmatrix}\\
&\begin{pmatrix}
1&f_{1}^{\pm}(u)&\ldots&*\\&\ddots\\&&\ddots&f_{n-2}^{\pm}(u)\\0&&&1
\end{pmatrix}\nonumber
\end{align}

Similarly restricting \eqref{e:2.5} and \eqref{e:2.6} to $e_{ij}\otimes e_{kl},2\leq
i,j,k,l\leq n$, we derive that
\begin{align}
\bar
R_{n-1}(u-v)\tilde{J}_1^{\pm}(u)\tilde{J}_2^{\pm}(v)=\tilde{J}_2^{\pm}(v)\tilde{J}_1^{\pm}(u)\bar
R_{n-1}(u-v)\end{align}
\begin{align}
\bar
R_{n-1}(u_{-}-v_{+})\tilde{J}_1^{+}(u)\tilde{J}_2^{-}(v)=\tilde{J}_2^{-}(v)\tilde{J}_1^{+}(u)\bar
R_{n-1}(u_{+}-v_{-})\end{align}
\begin{align}
\tilde{J}^{\pm}(u)=&\begin{pmatrix}
1&&&0\\e_{2}^{\pm}(u)&\ddots\\\vdots&&\ddots\\*&\ldots&e_{n-1}^{\pm}(u)&1
\end{pmatrix}
\begin{pmatrix}
k_{2}^{\pm}(u)&&&0\\&\ddots\\&&\ddots\\0&&&k_{n}^{\pm}(u)
\end{pmatrix}\\
&\begin{pmatrix}
1&f_{2}^{\pm}(u)&\ldots&*\\&\ddots\\&&\ddots&f_{n-1}^{\pm}(u)\\0&&&1
\end{pmatrix}\nonumber
\end{align}

By induction, we only need to check the relations between $e_{1}^{\pm}(u),k_{1}^{\pm}(u),f_{1}^{\pm}(u)$ and
$e_{n}^{\pm}(u),k_{n}^{\pm}(u),f_{n}^{\pm}(u)$. We now use the
formulas \eqref{e:2.14},\eqref{e:2.15} and  \eqref{e:2.17}. First we write down the matrices $L^{\pm}(u)$ and their inverse $L^{\pm}(u)^{-1}$,
\begin{align}
L^{\pm}(u)=\begin{pmatrix}
k_{1}^{\pm}(u)&k_{1}^{\pm}(u)e_{1}^{\pm}(u)&\ldots\\f_{1}^{\pm}(u)k_{1}^{\pm}(u)&\vdots&\ldots\\
\vdots&\vdots&\ldots
\end{pmatrix}
\end{align}
and
\begin{align}
L^{\pm}(u)^{-1}=\begin{pmatrix}
\ldots&\vdots&\vdots\\\ldots&\vdots&-e_{n-1}^{\pm}(u)k_{n}^{\pm}(u)^{-1}\\
\ldots&-k_{n-1}^{\pm}(u)^{-1}f_{1}^{\pm}(u)&k_{n}^{\pm}(u)^{-1}
\end{pmatrix}
\end{align}

From \eqref{e:2.14}, \eqref{e:2.15} and  \eqref{e:2.17}, we get
\begin{align}\label{e:2.111}
k_{1}^{\pm}(u)k_{n}^{\pm}(v)=k_{n}^{\pm}(v)k_{1}^{\pm}(u)\end{align}
\begin{align}\label{e:2.112}
k_{1}^{\pm}(u)f_{n-1}^{\pm}(v)=f_{n-1}^{\pm}(v)k_{1}^{\pm}(u)\end{align}
\begin{align}\label{e:2.113}
k_{1}^{\pm}(u)e_{n-1}^{\pm}(v)=e_{n-1}^{\pm}(v)k_{1}^{\pm}(u)\end{align}
\begin{align}\label{e:2.114}
k_{n}^{\pm}(u)f_{1}^{\pm}(v)=f_{1}^{\pm}(v)k_{n}^{\pm}(u)\end{align}
\begin{align}\label{e:2.115}
k_{n}^{\pm}(u)e_{1}^{\pm}(v)=e_{1}^{\pm}(v)k_{n}^{\pm}(u)\end{align}
\begin{align}\label{e:2.116}
k_{1}^{+}(u)k_{n}^{-}(v)=k_{n}^{-}(v)k_{1}^{+}(u)\end{align}
\begin{align}\label{e:2.117}
\frac{(u_{+}-v_{-})^{2}}{(u_{+}-v_{-})^{2}-h^{2}}k_{1}^{-}(u)k_{n}^{+}(v)=\frac{(u_{-}-v_{+})^{2}}{(u_{-}-v_{+})^{2}-h^{2}}k_{n}^{+}(v)k_{1}^{-}(u)\end{align}
\begin{align}\label{e:2.118}
k_{n}^{\pm}(u)f_{1}^{\mp}(v)=f_{1}^{\mp}(v)k_{n}^{\pm}(u)\end{align}
\begin{align}\label{e:2.119}
k_{n}^{\pm}(u)e_{1}^{\mp}(v)=e_{1}^{\mp}(v)k_{n}^{\pm}(u)\end{align}
\begin{align}\label{e:2.120}
k_{1}^{\pm}(u)f_{n-1}^{\mp}(v)=f_{n-1}^{\mp}(v)k_{1}^{\pm}(u)\end{align}
\begin{align}\label{e:2.121}
k_{1}^{\pm}(u)e_{n-1}^{\mp}(v)=e_{n-1}^{\mp}(v)k_{1}^{\pm}(u)\end{align}
\begin{align}\label{e:2.122}
e_{1}^{\pm}(u)e_{n-1}^{\pm}(v)=e_{n-1}^{\pm}(v)e_{1}^{\pm}(u)\end{align}
\begin{align}\label{e:2.123}
e_{1}^{\pm}(u)f_{n-1}^{\pm}(v)=f_{n-1}^{\pm}(v)e_{1}^{\pm}(u)\end{align}
\begin{align}\label{e:2.124}
e_{1}^{\pm}(u)e_{n-1}^{\mp}(v)=e_{n-1}^{\mp}(v)e_{1}^{\pm}(u)\end{align}
\begin{align}\label{e:2.125}
e_{1}^{\pm}(u)f_{n-1}^{\mp}(v)=f_{n-1}^{\mp}(v)e_{1}^{\pm}(u)\end{align}
\begin{align}\label{e:2.126}
f_{1}^{\pm}(u)e_{n-1}^{\pm}(v)=e_{n-1}^{\pm}(v)f_{1}^{\pm}(u)\end{align}
\begin{align}\label{e:2.127}
f_{1}^{\pm}(u)e_{n-1}^{\mp}(v)=e_{n-1}^{\mp}(v)f_{1}^{\pm}(u)\end{align}
\begin{align}\label{e:2.128}
f_{1}^{\pm}(u)f_{n-1}^{\pm}(v)=f_{n-1}^{\pm}(v)f_{1}^{\pm}(u)\end{align}
\begin{align}\label{e:2.129}
f_{1}^{\pm}(u)f_{n-1}^{\mp}(v)=f_{n-1}^{\mp}(v)f_{1}^{\pm}(u)\end{align}

From \eqref{e:2.113} and \eqref{e:2.121}, we get
\begin{align*}
k_{1}^{\pm}(u)^{-1}X_{n-1}^{+}(v)k_{1}^{\pm}(u)=X_{n-1}^{+}(v).\end{align*}

From \eqref{e:2.112} and \eqref{e:2.120}, we get
\begin{align*}
k_{1}^{\pm}(u)X_{n-1}^{-}(v)k_{1}^{\pm}(u)^{-1}=X_{n-1}^{-}(v).\end{align*}

From \eqref{e:2.122} and \eqref{e:2.124}, we get
\begin{align*}
X_{1}^{+}(u)X_{n-1}^{+}(v)=X_{n-1}^{+}(v)X_{1}^{+}(u).\end{align*}

From \eqref{e:2.128} and \eqref{e:2.129}, we get
\begin{align*}
X_{1}^{-}(u)X_{n-1}^{-}(v)=X_{n-1}^{-}(v)X_{1}^{-}(u).\end{align*}

This completes the proof of all the relations in the general case.
\end{proof}

\begin{corollary}
The following relations hold in the algebra $\DY_{h}(\mathfrak{sl}_n)$:
\begin{align*}
[H_{i}^{\pm}(u),H_{j}^{\pm}(v)]=0,
\end{align*}
\begin{align*}
(u_{\mp}-v_{\pm}+hB_{ij})(u_{\pm}-v_{\mp}-hB_{ij})H_{i}^{\pm}(u)H_{j}^{\mp}(v)\\
=(u_{\mp}-v_{\pm}-hB_{ij})(u_{\pm}-v_{\mp}+hB_{ij})H_{j}^{\mp}(v)H_{i}^{\pm}(u),
\end{align*}
\begin{align*}
H_{i}^{\pm}(u)^{-1}E_{j}(v)H_{i}^{\pm}(u)=\frac{u_{\pm}-v-hB_{ij}}{u_{\pm}-v+hB_{ij}}E_{j}(v),\\
H_{i}^{\pm}(u)F_{j}(v)H_{i}^{\pm}(u)^{-1}=\frac{u_{\mp}-v-hB_{ij}}{u_{\mp}-v+hB_{ij}}F_{j}(v),
\end{align*}
\begin{align*}
(u-v-hB_{ij})E_{i}(u)E_{j}(v)=(u-v+hB_{ij})E_{j}(v)E_{i}(u),\\
(u-v+hB_{ij})F_{i}(u)F_{j}(v)=(u-v-hB_{ij})F_{j}(v)F_{i}(u),
\end{align*}
\begin{align*}
\sum^{m=1-a_{ij}}_{\sigma\in\mathfrak{S}_{m}}[E_{i}(u_{\sigma(1)}),[E_{i}(u_{\sigma(2)})\cdots,[E_{i}(u_{\sigma(m)}),E_{j}(v)]\cdots]=0
~~ (i\neq j),\\
\sum^{m=1-a_{ij}}_{\sigma\in\mathfrak{S}_{m}}[F_{i}(u_{\sigma(1)}),[F_{i}(u_{\sigma(2)})\cdots,[F_{i}(u_{\sigma(m)}),F_{j}(v)]\cdots]=0
~~ (i\neq j),\\
[E_{i}(u),F_{j}(v)]=\frac{1}{h}\delta_{ij}\{\delta(u_{-}-v_{+})H_{i}^{+}(u_{-})-\delta(u_{+}-v_{-})H_{i}^{-}(v_{-})\}.
\end{align*}
Here we set $B_{ij}=\frac{1}{2}a_{ij}$,where $A=(a_{ij})$ is the
Cartan matrix of the Lie algebra $\mathfrak{sl}_n$.
\end{corollary}
\begin{proof}
Since $k_{i}^{\pm}(u)k_{j}^{\pm}(v)=k_{j}^{\pm}(v)k_{i}^{\pm}(u)$, we
have
\begin{align*}
k_{i}^{\pm}(u)k_{j}^{\pm}(v)^{-1}&=k_{j}^{\pm}(v)^{-1}k_{i}^{\pm}(u),\\
k_{i}^{\pm}(u)^{-1}k_{j}^{\pm}(v)&=k_{j}^{\pm}(v)k_{i}^{\pm}(u)^{-1},\\
k_{i}^{\pm}(u)^{-1}k_{j}^{\pm}(v)^{-1}&=k_{j}^{\pm}(v)^{-1}k_{i}^{\pm}(u)^{-1},
\end{align*}
so we get $[H_{i}^{\pm}(u),H_{j}^{\pm}(v)]=0$. Next, we prove that
\begin{align*}
(u_{\mp}-v_{\pm}+hB_{ij})(u_{\pm}-v_{\mp}-hB_{ij})H_{i}^{\pm}(u)H_{j}^{\mp}(v)\\
=(u_{\mp}-v_{\pm}-hB_{ij})(u_{\pm}-v_{\mp}+hB_{ij})H_{j}^{\mp}(v)H_{i}^{\pm}(u),
\end{align*}
Case 1:\quad $i=j,B_{ij}=1$
\begin{align}\label{e:2.130}
&\frac{u_{-}-v_{+}+h}{u_{-}-v_{+}}k_{i+1}^{+}(u+\frac{1}{2}hi)k_{i+1}^{-}(v+\frac{1}{2}hi)\\\nonumber
=&\frac{u_{+}-v_{-}+h}{u_{+}-v_{-}}k_{i+1}^{-}(v+\frac{1}{2}hi)k_{i+1}^{+}(u+\frac{1}{2}hi),\\ \label{e:2.131}
&\frac{u_{-}-v_{+}+h}{u_{-}-v_{+}}k_{i}^{+}(u+\frac{1}{2}hi)k_{i}^{-}(v+\frac{1}{2}hi)^{-1}\\\nonumber
=&\frac{u_{+}-v_{-}+h}{u_{+}-v_{-}}k_{i}^{-}(v+\frac{1}{2}hi)^{-1}k_{i}^{+}(u+\frac{1}{2}hi),\\ \label{e:2.132}
&\frac{(v_{+}-u_{-})^{2}}{(v_{+}-u_{-})^{2}-h^{2}}k_{i+1}^{+}(u+\frac{1}{2}hi)k_{i}^{-}(v+\frac{1}{2}hi)^{-1}\\\nonumber
=&\frac{(v_{-}-u_{+})^{2}}{(v_{-}-u_{+})^{2}-h^{2}}k_{i}^{-}(v+\frac{1}{2}hi)^{-1}k_{i+1}^{+}(u+\frac{1}{2}hi).
\end{align}
From \eqref{e:2.130}, \eqref{e:2.131} and \eqref{e:2.132}, we get
\begin{align*}
(u_{-}-v_{+}+h)(u_{+}-v_{-}-h)&H_{i}^{+}(u)H_{i}^{-}(v)\\
&=(u_{-}-v_{+}-h)(u_{+}-v_{-}+h)H_{i}^{-}(v)H_{i}^{+}(u).\end{align*}
Swapping $u$ and $v$, we get
\begin{align*}
(u_{+}-v_{-}+h)(u_{-}-v_{+}-h)&H_{i}^{-}(u)H_{i}^{+}(v)\\
&=(u_{+}-v_{-}-h)(u_{-}-v_{+}+h)H_{i}^{+}(v)H_{i}^{-}(u).\end{align*}
Thus,
\begin{align*}
(u_{\mp}-v_{\pm}+h)(u_{\pm}-v_{\mp}-h)&H_{i}^{\pm}(u)H_{i}^{\mp}(v)\\
&=(u_{\mp}-v_{\pm}-h)(u_{\pm}-v_{\mp}+h)H_{i}^{\mp}(v)H_{i}^{\pm}(u).\end{align*}
Case 2:\quad $j=i+1~or~j=i-1,B_{ij}=-\frac{1}{2}$\\
First, we assume $j=i+1$,
\begin{align}\label{e:2.133}
&\frac{u_{+}-v_{-}+\frac{1}{2}h}{u_{+}-v_{-}-\frac{1}{2}h}k_{i+1}^{+}(u+\frac{1}{2}hi)k_{i+1}^{-}(v+\frac{1}{2}hi+\frac{1}{2}h)^{-1}\\\nonumber
=&\frac{u_{-}-v_{+}+\frac{1}{2}h}{u_{-}-v_{+}-\frac{1}{2}h}k_{i+1}^{-}(v+\frac{1}{2}hi+\frac{1}{2}h)^{-1}k_{i+1}^{+}(u+\frac{1}{2}hi).
\end{align}
From \eqref{e:2.133}, we get
\begin{align*}
(u_{-}-v_{+}-\frac{1}{2}h)(u_{+}-v_{-}+&\frac{1}{2}h)H_{i}^{+}(u)H_{i+1}^{-}(v)\\
&=(u_{-}-v_{+}+\frac{1}{2}h)(u_{+}-v_{-}-\frac{1}{2}h)H_{i+1}^{-}(v)H_{i}^{+}(u).\end{align*}
Similarly, we can prove
\begin{align*}
(u_{+}-v_{-}-\frac{1}{2}h)(u_{-}-v_{+}+&\frac{1}{2}h)H_{i}^{-}(u)H_{i+1}^{+}(v)\\
&=(u_{+}-v_{-}+\frac{1}{2}h)(u_{-}-v_{+}-\frac{1}{2}h)H_{i+1}^{+}(v)H_{i}^{-}(u).\end{align*}
So we have
\begin{align*}
(u_{\mp}-v_{\pm}-\frac{1}{2}h)(u_{\pm}-v_{\mp}+&\frac{1}{2}h)H_{i}^{\pm}(u)H_{i+1}^{\mp}(v)\\
&=(u_{\mp}-v_{\pm}+\frac{1}{2}h)(u_{\pm}-v_{\mp}-\frac{1}{2}h)H_{i+1}^{\mp}(v)H_{i}^{\pm}(u).\end{align*}
Similarly, we can prove
\begin{align*}
(u_{\mp}-v_{\pm}-\frac{1}{2}h)(u_{\pm}-v_{\mp}+&\frac{1}{2}h)H_{i}^{\pm}(u)H_{i-1}^{\mp}(v)\\
&=(u_{\mp}-v_{\pm}+\frac{1}{2}h)(u_{\pm}-v_{\mp}-\frac{1}{2}h)H_{i-1}^{\mp}(v)H_{i}^{\pm}(u).\end{align*}
Case 3:\quad $\mid i-j\mid>1,B_{ij}=0$\\
Suppose $j<i$, then we have $j<j+1<i<i+1$. It follows that
$H_{i}^{\pm}(u)H_{j}^{\mp}(v)=H_{j}^{\mp}(v)H_{i}^{\pm}(u)$.

Combining the above three cases, we derive the following relation
\begin{align*}
(u_{\mp}-v_{\pm}+hB_{ij})(u_{\pm}-v_{\mp}-hB_{ij})H_{i}^{\pm}(u)H_{j}^{\mp}(v)\\
=(u_{\mp}-v_{\pm}-hB_{ij})(u_{\pm}-v_{\mp}+hB_{ij})H_{j}^{\mp}(v)H_{i}^{\pm}(u).
\end{align*}
Next, we prove that
\begin{align*}
H_{i}^{\pm}(u)^{-1}E_{j}(v)H_{i}^{\pm}(u)=\frac{u_{\pm}-v-hB_{ij}}{u_{\pm}-v+hB_{ij}}E_{j}(v).
\end{align*}
Case 1:\quad $i=j,B_{ij}=1$
\begin{align}\label{e:2.134}
k_{i+1}^{\pm}(u+\frac{1}{2}hi)^{-1}E_{i}(v)k_{i+1}^{\pm}(u+\frac{1}{2}hi)=\frac{u_{\pm}-v-h}{u_{\pm}-v}E_{i}(v),\\ \label{e:2.135}
k_{i}^{\pm}(u+\frac{1}{2}hi)E_{i}(v)k_{i}^{\pm}(u+\frac{1}{2}hi)^{-1}=\frac{u_{\pm}-v}{u_{\pm}-v+h}E_{i}(v),
\end{align}
From \eqref{e:2.134} and \eqref{e:2.135}, we get
$$H_{i}^{\pm}(u)^{-1}E_{i}(v)H_{i}^{\pm}(u)=\frac{u_{\pm}-v-h}{u_{\pm}-v+h}E_{i}(v).$$
Case 2:\quad $j=i+1,B_{ij}=-\frac{1}{2}$
\begin{align}\label{e:2.136}
k_{i+1}^{\pm}(u+\frac{1}{2}hi)^{-1}E_{i+1}(v)k_{i+1}^{\pm}(u+\frac{1}{2}hi)=\frac{u_{\pm}-v+\frac{1}{2}h}{u_{\pm}-v-\frac{1}{2}h}E_{i+1}(v),\\ \label{e:2.137}
k_{i}^{\pm}(u+\frac{1}{2}hi)E_{i+1}(v)k_{i}^{\pm}(u+\frac{1}{2}hi)^{-1}=E_{i+1}(v),
\end{align}
From \eqref{e:2.136} and \eqref{e:2.137}, we get
$$H_{i}^{\pm}(u)^{-1}E_{i+1}(v)H_{i}^{\pm}(u)=\frac{u_{\pm}-v+\frac{1}{2}h}{u_{\pm}-v-\frac{1}{2}h}E_{i+1}(v).$$
Case 3:\quad $j=i-1,B_{ij}=-\frac{1}{2}$
\begin{align} \label{e:2.138}
k_{i+1}^{\pm}(u+\frac{1}{2}hi)^{-1}E_{i-1}(v)k_{i+1}^{\pm}(u+\frac{1}{2}hi)=E_{i-1}(v),\\ \label{e:2.139}
k_{i}^{\pm}(u+\frac{1}{2}hi)E_{i-1}(v)k_{i}^{\pm}(u+\frac{1}{2}hi)^{-1}=\frac{u_{\pm}-v+\frac{1}{2}h}{u_{\pm}-v-\frac{1}{2}h}E_{i-1}(v),
\end{align}
From \eqref{e:2.138} and \eqref{e:2.139}, we get
$$H_{i}^{\pm}(u)^{-1}E_{i-1}(v)H_{i}^{\pm}(u)=\frac{u_{\pm}-v+\frac{1}{2}h}{u_{\pm}-v-\frac{1}{2}h}E_{i-1}(v).$$
Case 4:\quad $\mid i-j\mid>1,B_{ij}=0$
\begin{align}\label{e:2.140}
k_{i+1}^{\pm}(u+\frac{1}{2}hi)^{-1}E_{j}(v)k_{i+1}^{\pm}(u+\frac{1}{2}hi)=E_{j}(v),\\ \label{e:2.141}
k_{i}^{\pm}(u+\frac{1}{2}hi)E_{j}(v)k_{i}^{\pm}(u+\frac{1}{2}hi)^{-1}=E_{j}(v),
\end{align}
From \eqref{e:2.140} and \eqref{e:2.141}, we get
$$H_{i}^{\pm}(u)^{-1}E_{j}(v)H_{i}^{\pm}(u)=E_{j}(v).$$
Combining the above four cases, we derive the following relation
\begin{align*}
H_{i}^{\pm}(u)^{-1}E_{j}(v)H_{i}^{\pm}(u)=\frac{u_{\pm}-v-hB_{ij}}{u_{\pm}-v+hB_{ij}}E_{j}(v).
\end{align*}
Similarly, we can prove
\begin{align*}
H_{i}^{\pm}(u)F_{j}(v)H_{i}^{\pm}(u)^{-1}=\frac{u_{\mp}-v-hB_{ij}}{u_{\mp}-v+hB_{ij}}F_{j}(v).
\end{align*}
From the following three relations:
\begin{align*}
&(u-v-h)X_{i}^{+}(u)X_{i}^{+}(v)=(u-v+
h)X_{i}^{+}(v)X_{i}^{+}(u),\\
&(u-v+h)X_{i}^{+}(u)X_{i+1}^{+}(v)=(u-v)X_{i+1}^{+}(v)X_{i}^{+}(u),\\
&X_{i}^{+}(u)X_{j}^{+}(v)=X_{j}^{+}(v)X_{i}^{+}(u)\quad if~ |i-j|>1,
\end{align*}
we get
$$(u-v-hB_{ij})E_{i}(u)E_{j}(v)=(u-v+hB_{ij})E_{j}(v)E_{i}(u).$$
Similarly, we can prove
$$(u-v+hB_{ij})F_{i}(u)F_{j}(v)=(u-v-hB_{ij})F_{j}(v)F_{i}(u).$$
From the following two relations:
\begin{align*}
&X_{i}^{+}(u_{1})X_{i}^{+}(u_{2})X_{j}^{+}(v)-2X_{i}^{+}(u_{1})X_{j}^{+}(v)X_{i}^{+}(u_{2})+X_{j}^{+}(v)X_{i}^{+}(u_{1})X_{i}^{+}(u_{2})\\&+\{u_{1}\leftrightarrow
u_{2}\}=0 \quad if~ |i-j|=1,\\
&X_{i}^{+}(u)X_{j}^{+}(v)=X_{j}^{+}(v)X_{i}^{+}(u)\quad if~ |i-j|>1,
\end{align*}
we get
$$\sum_{\sigma\in\mathfrak{S}_{m}}[E_{i}(u_{\sigma(1)}),[E_{i}(u_{\sigma(2)})\cdots,[E_{i}(u_{\sigma(m)}),E_{j}(v)]\cdots]=0
~~ i\neq j,m=1-a_{ij}.$$ Similarly, we can prove
$$\sum_{\sigma\in\mathfrak{S}_{m}}[F_{i}(u_{\sigma(1)}),[F_{i}(u_{\sigma(2)})\cdots,[F_{i}(u_{\sigma(m)}),F_{j}(v)]\cdots]=0
~~ i\neq j,m=1-a_{ij}.$$ From
\begin{multline*}
[X_{i}^{+}(u),X_{j}^{-}(v)]\\
=h\delta_{ij}\{\delta(u_{-}-v_{+})k_{i+1}^{+}(u_{-})k_{i}^{+}(u_{-})^{-1}-\delta(u_{+}-v_{-})k_{i+1}^{-}(v_{-})k_{i}^{-}(v_{-})^{-1}\},
\end{multline*}
we get
$$[E_{i}(u),F_{j}(v)]=\frac{1}{h}\delta_{ij}\{\delta(u_{-}-v_{+})H_{i}^{+}(u_{-})-\delta(u_{+}-v_{-})H_{i}^{-}(v_{-})\}.$$
\end{proof}

\begin{remark}
Let $f(u)$ be a scalar function, and $R(u)=f(u)\bar R(u)$, Corollary
2.6 still holds if we change the normalization of R-matrix in
Definition 2.1 to $R(u)$. The relations in Corollary 2.6 were announced in Iohara's paper \cite{I}, while we provide a complete proof in this section.
\end{remark}

\section{Drinfeld realization of the Yangian double}\label{s:drins}
In this section, we will describe the Drinfeld realization of
$\DY_h(\mathfrak{gl}_n)$ and $\DY_h(\mathfrak{sl}_n)$.
First of all, we define a natural ascending filtration on the Yangian double $\DY_h(\mathfrak{gl}_n)$ by setting ${\rm deg}\,l_{ij}^{(r)}=r-1,$ ${\rm deg}\,l_{ij}^{(-r)}=-r$ for all $r\geqslant 1$ and ${\rm deg}\,c={\rm deg}\,h=0$. Denote by $\bar l_{ij}^{\ts(\pm r)}$ the image of $l_{ij}^{\ts(\pm r)}$ in the $(r-1)$ (or $(-r)$)-th component of the associated graded algebra ${\rm gr}\,\DY_h(\mathfrak{gl}_n)$. Let $\hat{\mathfrak{gl}}_{n}$ be  the central extension $\mathfrak{gl}_{n}[x,x^{-1}]\oplus\mathbb{C}K$ defined by the commutation relations
\begin{align*}
[E_{ij}[r],E_{kl}[s]]=\delta_{kj}E_{il}[r+s]-\delta_{il}E_{kj}[r+s]+r\delta_{kj}\delta_{il}\delta_{r,-s}K,
\end{align*}
and the element $K$ is central.
One then has the following result for the graded algebras.

\begin{proposition}
The mapping
\begin{equation}
E_{ij}[r-1]\mapsto \bar l_{ij}^{\ts (r)},\quad E_{ij}[-r]\mapsto \bar l_{ij}^{\ts (-r)},\quad K\mapsto \bar c,\quad h\mapsto\bar{h}
\end{equation}
defines an isomorphism
\begin{equation}\label{iso1}
{\rm U}(\mathfrak{gl}_{n}[x,x^{-1}]\oplus \mathbb{C}K)[[h]]\rightarrow {\rm gr}\,\DY_h(\mathfrak{gl}_n),
\end{equation}
where $K$ is the central element.
\end{proposition}

\begin{proof}
Using the expansion
\begin{align*}
\frac{1}{u-v}=u^{-1}+u^{-2}v+u^{-3}v^{2}+\cdots,
\end{align*}
and taking the coefficient at $u^{-r}v^{-s}~(r,s\geq 1)$ and keeping the highest degree terms on both sides of the relation \eqref{e:2.6}
gives that
\beq
[\bar l_{ij}^{\ts (r)},\bar l_{kl}^{\ts (-s)}]=\begin{cases}
\delta_{kj}\bar l_{il}^{\ts (r-s-1)}-\delta_{il}\bar l_{kj}^{\ts (r-s-1)}+(r-1)\delta_{kj}\delta_{il}\delta_{r,s+1} \bar c, \quad r\leq s\\
\delta_{kj}\bar l_{il}^{\ts (r-s)}-\delta_{il}\bar l_{kj}^{\ts (r-s)}+(r-1)\delta_{kj}\delta_{il}\delta_{r,s+1} \bar c, \quad r>s
\end{cases}
\eeq
Similarly, the coefficients at $u^{-r}v^{-s}~(r,s\geq 1)$ and $u^{r}v^{s}~(r,s\geq 1)$ of \eqref{e:2.5} imply that
\begin{align}
[\bar l_{ij}^{\ts (r)},\bar l_{kl}^{\ts (s)}]&=\delta_{kj}\bar l_{il}^{\ts (r+s-2)}-\delta_{il}\bar l_{kj}^{\ts (r+s-2)},\\
[\bar l_{ij}^{\ts (-r)},\bar l_{kl}^{\ts (-s)}]&=\delta_{kj}\bar l_{il}^{\ts (-r-s)}-\delta_{il}\bar l_{kj}^{\ts (-r-s)}.
\end{align}
It follows that the map \eqref{iso1} is a surjective homomorphism. Note that the ordered monomials in the generators $l_{ij}^{(r)},l_{ij}^{(-r)}$ and $c$ form a topological
basis of the algebra $\DY_h(\mathfrak{gl}_n)$, see \cite{ji:cen}. This implies that the map is also injective.
\end{proof}

Next we introduce the $ij$-th quasideterminant of a matrix.
Let $A=[a_{ij}]$ be an $N\times N$ matrix over a ring with $1$. Delete the $i$-th row
and $j$-th column of $A$, we obtain a submatrix of $A$. We will denote it by $A^{ij}$.
If the matrix
$A^{ij}$ is invertible, we define the $ij$-{\em th quasideterminant of} $A$
by the following formula
\ben
|A|_{ij}=a_{ij}-r^{\tss j}_i(A^{ij})^{-1}\ts c^{\tss i}_j,
\een
where $r^{\tss j}_i$ is the row matrix obtained from the $i$-th
row of $A$ by deleting the element $a_{ij}$, and $c^{\tss i}_j$
is the column matrix obtained from the $j$-th
column of $A$ by deleting the element $a_{ij}$; see
\cite{gr:dm}. We also denote the quasideterminant $|A|_{ij}$
by boxing the entry $a_{ij}$,
\ben
|A|_{ij}=\left|\begin{matrix}a_{11}&\dots&a_{1j}&\dots&a_{1N}\\
                                   &\dots&      &\dots&      \\
                             a_{i1}&\dots&\boxed{a_{ij}}&\dots&a_{iN}\\
                                   &\dots&      &\dots&      \\
                             a_{N1}&\dots&a_{Nj}&\dots&a_{NN}
                \end{matrix}\right|.
\een

Assume the Gauss decomposition of the generator matrix $L^{\pm}(u)$ is $L^{\pm}(u)=F^{\pm}(u)H^{\pm}(u)E^{\pm}(u)$, then we have the well-known formulas for the entries of the matrices $H^{\pm}(u),E^{\pm}(u)$ and $F^{\pm}(u)$; see e.g. \cite[Sec.~1.11]{Mo1}.

\begin{proposition}
For $i=1,\dots,n,$
\begin{align}
k^{\pm}_i(u)&=qdetL_{i,i}^{\pm}(u-(i-1)h)qdetL_{i-1,i-1}^{\pm}(u-(i-1)h)^{-1}\\ \nonumber
&=\begin{vmatrix} l^{\pm}_{1\tss 1}(u)&\dots&l^{\pm}_{1\ts i-1}(u)&l^{\pm}_{1\tss i}(u)\\
                          \vdots&\ddots&\vdots&\vdots\\
                         l^{\pm}_{i-1\ts 1}(u)&\dots&l^{\pm}_{i-1\ts i-1}(u)&l^{\pm}_{i-1\ts i}(u)\\
                         l^{\pm}_{i\tss 1}(u)&\dots&l^{\pm}_{i\ts i-1}(u)&\boxed{l^{\pm}_{i\tss i}(u)}\\
           \end{vmatrix}.
\end{align}
For $1\leqslant i<j\leqslant n$,
\begin{align}
e^{\pm}_{ij}(u)&=qdetL_{i,i}^{\pm}(u-(i-1)h)^{-1}qdetL_{i,j}^{\pm}(u-(i-1)h)\\ \nonumber
&=k^{\pm}_i(u)^{-1}\ts\begin{vmatrix} l^{\pm}_{1\tss 1}(u)&\dots&l^{\pm}_{1\ts i-1}(u)&l^{\pm}_{1\ts j}(u)\\
                          \vdots&\ddots&\vdots&\vdots\\
                         l^{\pm}_{i-1\ts 1}(u)&\dots&l^{\pm}_{i-1\ts i-1}(u)&l^{\pm}_{i-1\ts j}(u)\\
                         l^{\pm}_{i\tss 1}(u)&\dots&l^{\pm}_{i\ts i-1}(u)&\boxed{l^{\pm}_{i\tss j}(u)}\\
           \end{vmatrix}
\end{align}
and
\begin{align}
f^{\pm}_{ji}(u)&=qdetL_{j,i}^{\pm}(u-(i-1)h)qdetL_{i, i}^{\pm}(u-(i-1)h)^{-1}\\ \nonumber
&=\begin{vmatrix} l^{\pm}_{1\tss 1}(u)&\dots&l^{\pm}_{1\ts i-1}(u)&l^{\pm}_{1\tss i}(u)\\
                          \vdots&\ddots&\vdots&\vdots\\
                         l^{\pm}_{i-1\ts 1}(u)&\dots&l^{\pm}_{i-1\ts i-1}(u)&l^{\pm}_{i-1\ts i}(u)\\
                         l^{\pm}_{j\ts 1}(u)&\dots&l^{\pm}_{j\ts i-1}(u)&\boxed{l^{\pm}_{j\tss i}(u)}\\
           \end{vmatrix}\ts k^{\pm}_i(u)^{-1}.
\end{align}
\end{proposition}

The coefficients of the matrices are referred to
as the Gaussian generators, explicitly
\ben
e^{+}_{ij}(u)=h\sum_{r=1}^{\infty} e_{ij}^{(r)}\tss u^{-r},
f^{+}_{ji}(u)=h\sum_{r=1}^{\infty} f_{ji}^{(r)}\tss u^{-r},
k^{+}_i(u)=1+h\sum_{r=1}^{\infty} k_i^{(r)}\tss u^{-r},
\een
\ben
e^{-}_{ij}(u)=-h\sum_{r=1}^{\infty} e_{ij}^{(-r)}\tss u^{r-1},
f^{-}_{ji}(u)=-h\sum_{r=1}^{\infty} f_{ji}^{(-r)}\tss u^{r-1},
k^{-}_i(u)=1-h\sum_{r=1}^{\infty} k_i^{(-r)}\tss u^{r-1}.
\een
Now we are in a position to give the Drinfeld realization of
$\DY_h(\mathfrak{gl}_n)$ and $\DY_h(\mathfrak{sl}_n)$.

\begin{theorem}
The Yangian double $\DY_h(\mathfrak{gl}_n)$ is topologically generated by the coefficients of the series $k^{\pm}_i(u)(i=1,\dots,n)$, $e^{\pm}_{j}(u)$ and $f^{\pm}_{j}(u)(j=1,\dots,n-1)$ and the central element $c$ subject to the defining relations in Theorem 2.5, where the indices run through all admissible values.
\end{theorem}

\begin{proof}
Let $\wh \DY_h(\mathfrak{gl}_n)$ be the algebra
with generators and relations as in the statement of the theorem. Theorem 2.5 implies that there is a homomorphism $\phi:\wh \DY_h(\mathfrak{gl}_n)\rightarrow \DY_h(\mathfrak{gl}_n)$ which takes
the generators $k_{i}^{(r)}$,
$e_{i}^{(r)}$, $f_{i}^{(r)}$ and $c$ of $\wh \DY_h(\mathfrak{gl}_n)$ to the corresponding elements
of $\DY_h(\mathfrak{gl}_n)$. To prove the surjectivity of the map $\phi$, we only need to show that $e^{\pm}_{1,n}(u)$ and $f^{\pm}_{n,1}(u)$ are generated by $k^{\pm}_i(u),e^{\pm}_j(u)$ and $f^{\pm}_j(u)$ in the algebra $\DY_h(\mathfrak{gl}_n)$ where $i=1,\dots,n,j=1,\dots,n-1$. Since all other elements $e^{\pm}_{ij}(u)$ and $f^{\pm}_{ji}(u)$ are generated by $k^{\pm}_i(u),e^{\pm}_j(u)$ and $f^{\pm}_j(u)$ by induction. From \eqref{e:2.14} and \eqref{e:2.15}, we can get the relations between $e^{\pm}_{1,n-1}(u)$ and $e^{\pm}_{n-1,n}(u)$, and the relations between $f^{\pm}_{n-1,1}(u)$ and $f^{\pm}_{n,n-1}(u)$, which also contain $e^{\pm}_{1,n}(u)$ and $f^{\pm}_{n,1}(u)$. These formulas are similar to  \eqref{e:2.87}, \eqref{e:2.88},(2.96) and (2.97) in the case when $n=3$. It follows that $e^{\pm}_{1,n}(u)$ and $f^{\pm}_{n,1}(u)$ are generated by $k^{\pm}_i(u),e^{\pm}_j(u)$ and $f^{\pm}_j(u)$. Thus we have proved that $\phi$ is surjective. Next we show that $\phi$ is injective. We start by showing that the set of monomials in
\ben
k_i^{(r)}, \quad with \quad i=1,\dots,n,r\in\mathbb{Z}^{\times},
\een
and
\ben
e_{ij}^{(r)},\quad f_{ji}^{(r)}, \quad with \quad 1\leq i<j\leq n,r\in\mathbb{Z}^{\times},
\een
and $c$ taken in some fixed order is linearly independent in the Yangian double $\DY_h(\mathfrak{gl}_n)$. Applying Proposition 3.1 to the matrix $T^{+}(u)$, we deduce that the images of the elements $k_i^{(r)},e_{ij}^{(r)}$ and $f_{ji}^{(r)}(r\geq 1)$ in the $(r-1)$-th component of the graded algebra ${\rm gr}\,\DY_h(\mathfrak{gl}_n)$ under the isomorphism \eqref{iso1} respectively correspond to the elements $E_{ii}\ts x^{r-1},E_{ij}\ts x^{r-1}$ and $E_{ji}\ts x^{r-1}$. Similarly, the images of the elements $k_i^{(-r)},e_{ij}^{(-r)}$ and $f_{ji}^{(-r)}(r\geq 1)$ in the $(-r)$-th component of the graded algebra ${\rm gr}\,\DY_h(\mathfrak{gl}_n)$ under the isomorphism \eqref{iso1} respectively correspond to the elements $E_{ii}\ts x^{-r},E_{ij}\ts x^{-r}$ and $E_{ji}\ts x^{-r}$. Hence the claim follows from the Poincar\'e-Birkhoff--Witt theorem for ${\rm U}(\mathfrak{gl}_{n}[x,x^{-1}]\oplus \mathbb{C}K)$.

For any $1\leq i<j\leq n,r\in\mathbb{Z}^{\times}$, define elements $e_{ij}^{(r)}$ and $f_{ji}^{(r)}$ of $\wh \DY_h(\mathfrak{gl}_n)$ inductively by the relations $e_{i,i+1}^{(r)}=e_{i}^{(r)},f_{i+1,i}^{(r)}=f_{i}^{(r)}$ and
\ben
e_{i,j+1}^{(r)}=[e_{ij}^{(r)},e_{j}^{(1)}],\quad f_{j+1,i}^{(r)}=[f_{j}^{(1)},f_{ji}^{(r)}]\quad for \quad j>i,r\in\mathbb{Z}^{\times}.
\een
Obviously, these relations are consistent with those in $\DY_h(\mathfrak{gl}_n)$. The injectivity of $\phi$ will follow if we prove that the algebra $\wh \DY_h(\mathfrak{gl}_n)$ is spanned by the monomials in $k_i^{(r)},e_{ij}^{(r)}, f_{ji}^{(r)}$ and $c$ taken in some fixed order. To see this, we introduce some notations. Denote by $\wh \Ec$, $\wh \Fc$
and $\wh \Hc$ the subalgebras of $\wh \DY_h(\mathfrak{gl}_n)$ respectively
generated by all elements
of the form $e_{i}^{(r)}$, $f_{i}^{(r)}$ and $k_{i}^{(r)}.$
Denote by $\wh \Ec^{+}$, $\wh \Fc^{+}$
and $\wh \Hc^{+}$ the subalgebras of $\wh \DY_h(\mathfrak{gl}_n)$ respectively
generated by all elements
of the form $e_{i}^{(r)}$, $f_{i}^{(r)}$ and $k_{i}^{(r)}$ with $r>0.$
Denote by $\wh \Ec^{-}$, $\wh \Fc^{-}$
and $\wh \Hc^{-}$ the subalgebras of $\wh \DY_h(\mathfrak{gl}_n)$ respectively
generated by all elements
of the form $e_{i}^{(r)}$, $f_{i}^{(r)}$ and $k_{i}^{(r)}$ with $r<0.$
Define an ascending filtration
on $\wh \Ec^{-}$ by setting $\deg e_{i}^{(-r)}=-r, \deg h=0$.
Denote by $\gr\wh \Ec^{-}$ the corresponding graded algebra.
Let $\eb_{ij}^{(-r)}$ be the image of $e_{ij}^{(-r)}$ in the
$(-r)$-th component of the graded algebra $\gr\wh \Ec^{-}$.
Note that the algebra $\wh \Ec^{+}$ is spanned by the monomials in the elements $e_{ij}^{(r)}(r>0)$ taken in some fixed order; see \cite[Sec.~3.1]{Mo1}.
Similarly, the desired spanning property of the algebra $\wh \Ec^{-}$ follows from the relations
\beq\label{ebrela}
[\eb_{ij}^{(-r)},\eb_{kl}^{(-s)}]=\delta_{kj}\eb_{il}^{(-r-s)}-\delta_{il}\eb_{kj}^{(-r-s)}.
\eeq
We can use the same method in \cite[Sec.~3.1]{Mo1} to prove \eqref{ebrela}. In addition, we can swap $e_{ij}^{(r)}$ and $e_{kl}^{(-s)}$ ($r,s>0$) by the relations between $e^{+}_{ij}(u)$ and $e^{-}_{kl}(v)$. These relations can be obtained from (2.52), (2.99), (2.124) and $e_{i,j+1}^{(r)}=[e_{ij}^{(r)},e_{j}^{(1)}]$ by induction. For example, if we want to swap $e_{13}^{(r)}$ and $e_{2}^{(-s)}$, write $e_{13}^{(r)}$ in the form $[e_{1}^{(r)},e_{2}^{(1)}]$ firstly. Then we swap $e_{1}^{(r)}$ and $e_{2}^{(-s)}$, $e_{2}^{(1)}$ and $e_{2}^{(-s)}$ with the relations between $e_{1}^{+}(u)$ and $e_{2}^{-}(v)$, $e_{2}^{+}(u)$ and $e_{2}^{-}(v)$ respectively. After these steps, we can write
the product $e_{13}^{(r)}e_{2}^{(-s)}$ or $e_{2}^{(-s)}e_{13}^{(r)}$ as the ordered monomials in the elements $e_{ij}^{(r)}$ and $e_{ij}^{(-r)}$.
Therefore, the algebra $\wh \Ec$ is spanned by the ordered monomials in the elements $e_{ij}^{(r)}$ and $e_{ij}^{(-r)}$ taken in some fixed order. The same is true for $\wh \Fc$. Observe that the ordered monomials in the $k_i^{(r)}$ and the $k_i^{(-r)}$ span $\wh\Hc$. Moreover, the defining relations of $\wh \DY_h(\mathfrak{gl}_n)$ implies that the multiplication map
\beq
\wh\Fc~\widetilde{\otimes}~\wh\Hc~\widetilde{\otimes}~\wh\Ec~\widetilde{\otimes}~\Ac c  \rightarrow \wh \DY_h(\mathfrak{gl}_n)
\eeq
is surjective. Here $A~\widetilde{\otimes}~B$ denotes the topological tensor product of the algebras $A$ and $B$, which is the $h$-adic completion of $A\otimes_{\mathbb{C}[[h]]}B$. Therefore, if we let
the elements of $\wh\Fc$ precede the elements
of $\wh\Hc$, and the latter precede the elements of $\wh \Ec$, $c$ included in the ordering in an
arbitrary way, then the ordered monomials in the set of elements $k_i^{(r)},e_{ij}^{(r)},f_{ji}^{(r)}$ and $c$ with $r\in \mathbb{Z}^{\times}$
span $\wh \DY_h(\mathfrak{gl}_n)$. It follows that
$\phi$ is injective.

\end{proof}

\begin{corollary}
The Yangian double $\DY_h(\mathfrak{sl}_n)$ is topologically generated by the Drinfeld generators
$\{h_{il},e_{il},f_{il}\mid i=1,\cdots,n-1,l\in\mathbb{Z}\}$ and the
central element $c$ subject to the defining relations in Corollary 2.6, where the Drinfeld currents are defined as
follows:
\begin{align*}
&H^{+}_{i}(u)=1+h\sum_{l\geq 0}h_{il}u^{-l-1},H^{-}_{i}(u)=1-h\sum_{l<
0}h_{il}u^{-l-1},\\
&E_{i}(u)=\sum_{l\in\mathbb{Z}}e_{il}u^{-l-1},F_{i}(u)=\sum_{l\in\mathbb{Z}}f_{il}u^{-l-1}.
\end{align*}
\end{corollary}

\begin{proof}
Let $\wh \DY_h(\mathfrak{sl}_n)$ be the algebra
with generators and relations as in the statement of the corollary. Corollary 2.6 implies that there is a surjective homomorphism $\varphi:\wh \DY_h(\mathfrak{sl}_n)\rightarrow \DY_h(\mathfrak{sl}_n)$ which takes
the generators $h_{il},e_{il},f_{il}$ and $c$ of $\wh \DY_h(\mathfrak{sl}_n)$ to the corresponding elements
of $\DY_h(\mathfrak{sl}_n)$. Denote by $\Ec$, $\Fc$
and $\Hc$ the subalgebras of $\DY_h(\mathfrak{gl}_n)$ respectively
generated by all elements
of the form $e_{ij}^{(r)}$, $f_{ji}^{(r)}$ and $k_{i}^{(r)}.$
By the decomposition
\beq
\DY_h(\mathfrak{sl}_n)= \Ec\widetilde{\otimes} (\DY_h(\mathfrak{sl}_n)\cap\Hc)\widetilde{\otimes} \Fc\widetilde{\otimes}\Ac c,
\eeq
the corresponding arguments of the proof of Theorem 3.1 gives that $\varphi$ is also injective. This completes the proof of the corollary.
\end{proof}

\newpage
\section{The center of the Yangian double at the critical level}\label{s:cent}
In this section, we will construct central elements of the Yangian
double at the critical level. First we need to normalize the Yang
R-matrix to satisfy the crossing symmetry condition. Let us define
$R(u)=f(u)\bar{R}(u)$, where $\bar{R}(u)$ is Yang's R-matrix
\eqref{1}, and $f(u)=1+\sum_{k=1}^{\infty}f_{k}u^{-k}$ is the function of
$u, h$ defined by the power series expansion and satisfies the
functional equation
\begin{equation}\label{e:funceq}
f(u-nh)=\frac{u^2-h^2}{u^2}f(u).
\end{equation}
Clearly the coefficients $f_{k}$ are rational functions in $h$
uniquely determined by  \eqref{e:funceq}. It can be seen that
\begin{equation}
f(u)=\prod_{k=1}^{\infty}(1-\frac{h^2}{(u+khn)^2}).
\end{equation}

Now we modify the defining relations of the Yangian double
$\DY_{h}(\mathfrak{gl}_n)$ as follows:
\begin{align} \label{e:4.3}
&[L^{\pm}(u),c]=0,\\ \label{e:4.4}
&R(u-v)L_1^{\pm}(u)L_2^{\pm}(v)=L_2^{\pm}(v)L_1^{\pm}(u)R(u-v),\\ \label{e:4.5}
&R(u-v-\frac{1}{2}hc)L_1^{+}(u)L_2^{-}(v)=L_2^{-}(v)L_1^{+}(u)R(u-v+\frac{1}{2}hc),
\end{align}
In the rest of the paper, the Yangian double defined here is still
denoted by $\DY_{h}(\mathfrak{gl}_n)$.

Denote by $\overline{\DY}_h(\mathfrak{gl}_n)_{cr}$ the Yangian
double at the critical level $c=-n$, which is the quotient of
${\DY}_h(\mathfrak{gl}_n)$ modulo the ideal generated by the
relation $c=-n$. Define its completion ${\DY}_h(\mathfrak{gl}_n)_{cr}$ as
the inverse limit
\begin{align*}
{\DY}_h(\mathfrak{gl}_n)_{cr}=\underleftarrow{\lim}\overline{\DY}_{h}(\mathfrak{gl}_n)_{cr}/J_{p},\quad
p>0
\end{align*}
where $J_{p}$ is the left ideal of
$\overline{\DY}_{h}(\mathfrak{gl}_n)_{cr}$ generated by all elements
$l_{ij}^{r}$ with $r\geq p$.

Let $\mathcal {A}[\mathfrak{S}_k]$ be the group algebra of the
symmetry
group $\mathfrak S_k$,
which naturally acts on $(\mathbb{C}^n)^{\otimes k}$ by permutation.
Let $A_k$ be the antisymmetrizer
\begin{equation}\label{8}
A_k=\frac{1}{k!}\sum_{\sigma\in\mathfrak{S}_k}(sgn\,\sigma)\sigma, 
\end{equation}
where $sgn\, \sigma$ is the sign of the permutation $\sigma$.

For each $k=1,\cdots,n$, introduce the Laurent series $\ell_k(u)$ in
$u$ by
\begin{align}\nonumber
\ell_k(u)&=tr_{1\cdots k}~A_kL_1^{-}(u_1)\cdots L_k^{-}(u_k)\cdot\\
\label{4} &\quad L_k^{+}(u_k+\frac{1}{2}hn)^{-1}\cdots
L_1^{+}(u_1+\frac{1}{2}hn)^{-1},
\end{align}
where $u_i=u+(i-1)h$ and the partial trace is taken over all $k$
copies of $End\,\mathbb{C}^n$ in $(End\,\mathbb{C}^n)^{\otimes
k}\otimes {\DY}_h(\mathfrak{gl}_n)_{cr}$, thus the coefficients of $\ell_k(u)$ belong to
${\DY}_h(\mathfrak{gl}_n)_{cr}$.

The following famous lemma of Jucys is well-known, see \cite{Mo1}.
\begin{lemma}\label{a} One has that
\begin{equation}
\prod_{1\leq i<j\leq k}R_{ij}(u_i-u_j)=k!~A_k,
\end{equation}
where the product is taken in the lexicographical order on the pairs
$(i,j)$.
\end{lemma}
Using Lemma \ref{a} and the RTT relations \eqref{e:4.4}, we can easily
get the following lemma.
\begin{lemma}\label{b} One has that
\begin{align}
A_kL_1^{-}(u_1)\cdots L_k^{-}(u_k)&=L_k^{-}(u_k)\cdots
L_1^{-}(u_1)A_k\\ \nonumber
A_kL_k^{+}(u_k+\frac{1}{2}hn)^{-1}&\cdots
L_1^{+}(u_1+\frac{1}{2}hn)^{-1}\\
&=L_1^{+}(u_1+\frac{1}{2}hn)^{-1}\cdots
L_k^{+}(u_k+\frac{1}{2}hn)^{-1}A_k
\end{align}
\end{lemma}

The following condition (crossing symmetry) is important for our later discussion.
\begin{lemma}
The R-matrix $R(u)$ satisfies the crossing symmetry relations:
\begin{align}
(R_{12}(u)^{-1})^{t_2}R_{12}(u-hn)^{t_2}&=I,\\
R_{12}(u-hn)^{t_1}(R_{12}(u)^{-1})^{t_1}&=I.
\end{align}
\end{lemma}

\begin{proof}
Set $Q=P^{t_1}=P^{t_2}$, it is easy to check that
$$Q^2=nQ.$$ 
then we have
\begin{align*}
\bar{R}(-u)^{t_2}\bar{R}(u-hn)^{t_2}=(I-\frac{h}{u}Q)(I+\frac{h}{u-hn}Q)~=~I.
\end{align*}
It follows that
\begin{align*}
&(R_{12}(u)^{-1})^{t_2}R_{12}(u-hn)^{t_2}=\frac{f(u-hn)}{f(u)}\frac{u^2}{u^2-h^2}\bar{R}(-u)^{t_2}\bar{R}(u-hn)^{t_2}=I.
\end{align*}
We can prove the second crossing symmetry relation similarly.
\end{proof}

\begin{theorem}\label{1.1}
The coefficients of $\ell_k(u)$ belong to the center of the
completed Yangian double at the critical level
${\DY}_h(\mathfrak{gl}_n)_{cr}$ for all $k=1,\cdots,n$.
\end{theorem}

\begin{proof}
Consider the tensor product $End\,\mathbb{C}^n\otimes(End\,\mathbb{C}^n)^{\otimes k}\otimes
{\DY}_h(\mathfrak{gl}_n)_{cr}$ and label the first copy of $End\,\mathbb{C}^n$ by $0$.
It suffices to show that
$\ell_k(u)$ commutes with $L_0^{+}(z)$. From the defining relation \eqref{e:4.5}, we get
\begin{align}
&L_0^{+}(z)L_k^{-}(u_k)\cdots
L_1^{-}(u_1)=R_{0k}(z-u_k+\frac{1}{2}hn)^{-1}\cdots
R_{01}(z-u_1+\frac{1}{2}hn)^{-1}\nonumber\\
&R_{01}(z-u_1+\frac{1}{2}hn)\cdots
R_{0k}(z-u_k+\frac{1}{2}hn)L_0^{+}(z)L_k^{-}(u_k)\cdots
L_1^{-}(u_1)\nonumber\\
&=R_{0k}(z-u_k+\frac{1}{2}hn)^{-1}\cdots
R_{01}(z-u_1+\frac{1}{2}hn)^{-1}L_k^{-}(u_k)\cdots
L_1^{-}(u_1)L_0^{+}(z)\nonumber\\
&R_{01}(z-u_1-\frac{1}{2}hn)\cdots
R_{0k}(z-u_k-\frac{1}{2}hn).
\end{align}
Relation \eqref{e:4.4} gives
\begin{multline}
L_0^{+}(z)R_{0a}(z-u_a-\frac{1}{2}hn)L_a^{+}(u_a+\frac{1}{2}hn)^{-1}\\
=L_a^{+}(u_a+\frac{1}{2}hn)^{-1}R_{0a}(z-u_a-\frac{1}{2}hn)L_0^{+}(z)
\end{multline}
for $a=1,\cdots,k.$ Therefore,
\begin{align}
&L_0^{+}(z)R_{01}(z-u_1-\frac{1}{2}hn)\cdots
R_{0k}(z-u_k-\frac{1}{2}hn)L_1^{+}(u_1+\frac{1}{2}hn)^{-1}\cdots\nonumber\\
&L_k^{+}(u_k+\frac{1}{2}hn)^{-1}=L_1^{+}(u_1+\frac{1}{2}hn)^{-1}\cdots
L_k^{+}(u_k+\frac{1}{2}hn)^{-1}R_{01}(z-u_1-\frac{1}{2}hn)\cdots\nonumber\\
&R_{0k}(z-u_k-\frac{1}{2}hn)L_0^{+}(z).
\end{align}
To conclude $L_0^{+}(z)L_k(u)=L_k(u)L_0^{+}(z)$, it is enough to prove
\begin{align}\label{7}
&tr_{1,\cdots,k}~R_{0k}(z-u_k+\frac{1}{2}hn)^{-1}\cdots
R_{01}(z-u_1+\frac{1}{2}hn)^{-1}L_k^{-}(u_k)\cdots L_1^{-}(u_1)\nonumber\\
&L_1^{+}(u_1+\frac{1}{2}hn)^{-1}\cdots
L_k^{+}(u_k+\frac{1}{2}hn)^{-1}R_{01}(z-u_1-\frac{1}{2}hn)\cdots\nonumber\\
&R_{0k}(z-u_k-\frac{1}{2}hn)A_k=\ell_k(u).
\end{align}
 The R-matrix $R(u)$ satisfies the Yang-Baxter equation
$$R_{12}(u-v)R_{13}(u)R_{23}(v)=R_{23}(v)R_{13}(u)R_{12}(u-v)$$
Therefore, Lemma \ref{a} implies
\begin{multline}\label{5}
R_{01}(z-u_1-\frac{1}{2}hn)\cdots
R_{0k}(z-u_k-\frac{1}{2}hn)A_k\\
=A_kR_{0k}(z-u_k-\frac{1}{2}hn)\cdots
R_{01}(z-u_1-\frac{1}{2}hn)
\end{multline}
\begin{multline}\label{6}
R_{0k}(z-u_k+\frac{1}{2}hn)^{-1}\cdots
R_{01}(z-u_1+\frac{1}{2}hn)^{-1}A_k\\
=A_kR_{01}(z-u_1+\frac{1}{2}hn)^{-1}\cdots
R_{0k}(z-u_k+\frac{1}{2}hn)^{-1}
\end{multline}
Using (\ref{5}), (\ref{6}) and Lemma \ref{b}, we can write
(\ref{7}) in the form
\begin{align}
&tr_{1,\cdots,k}~R_{0k}(z-u_k+\frac{1}{2}hn)^{-1}\cdots
R_{01}(z-u_1+\frac{1}{2}hn)^{-1}A_kL_1^{-}(u_1)\cdots L_k^{-}(u_k)\nonumber\\
&L_k^{+}(u_k+\frac{1}{2}hn)^{-1}\cdots
L_1^{+}(u_1+\frac{1}{2}hn)^{-1}R_{0k}(z-u_k-\frac{1}{2}hn)\cdots
R_{01}(z-u_1-\frac{1}{2}hn).
\end{align}
Now replace $A_k$ with $(A_k)^{2}$ and move one copy of $A_k$ to the left with help of
\eqref{6} and the other copy to the right. Then we get
$$tr_{1,\cdots,k}~A_kR_{01}(z-u_1+\frac{1}{2}hn)^{-1}\cdots
R_{0k}(z-u_k+\frac{1}{2}hn)^{-1}L_k^{-}(u_k)\cdots L_1^{-}(u_1)$$
$$L_1^{+}(u_1+\frac{1}{2}hn)^{-1}\cdots
L_k^{+}(u_k+\frac{1}{2}hn)^{-1}R_{01}(z-u_1-\frac{1}{2}hn)\cdots
R_{0k}(z-u_k-\frac{1}{2}hn)A_k.$$
Next we use the cyclic property of trace
to move the left copy of $A_k$ to the right-most position and
replace $(A_k)^{2}$ with $A_k$. After these
transformations, we obtain the following expression
\begin{align}
&tr_{1,\cdots,k}~R_{01}(z-u_1+\frac{1}{2}hn)^{-1}\cdots
R_{0k}(z-u_k+\frac{1}{2}hn)^{-1}A_kL_1^{-}(u_1)\cdots
L_k^{-}(u_k)\nonumber\\
&L_k^{+}(u_k+\frac{1}{2}hn)^{-1}\cdots
L_1^{+}(u_1+\frac{1}{2}hn)^{-1}R_{0k}(z-u_k-\frac{1}{2}hn)\cdots
R_{01}(z-u_1-\frac{1}{2}hn).
\end{align}
Now we let
\begin{align*}
&X=R_{01}(z-u_1+\frac{1}{2}hn)^{-1}\cdots
R_{0k}(z-u_k+\frac{1}{2}hn)^{-1}A_kL_1^{-}(u_1)\cdots
L_k^{-}(u_k)\\
&L_k^{+}(u_k+\frac{1}{2}hn)^{-1}\cdots
L_1^{+}(u_1+\frac{1}{2}hn)^{-1},\\
&Y=R_{0k}(z-u_k-\frac{1}{2}hn)\cdots R_{01}(z-u_1-\frac{1}{2}hn).
\end{align*}
Since $(A(I\otimes B))^{t_2}=(I\otimes B)^{t_2}A^{t_2}$ for all
$A\in (End\,\mathbb{C}^n)^{\otimes2}$,we have
\begin{align}
&X^{t_{1}\cdots t_{k}}=L^{t_{1}\cdots
t_{k}}(R_{01}(z-u_1+\frac{1}{2}hn)^{-1})^{t_1}\cdots
(R_{0k}(z-u_k+\frac{1}{2}hn)^{-1})^{t_k},\\
&Y^{t_{1}\cdots t_{k}}=(R_{0k}(z-u_k-\frac{1}{2}hn))^{t_k}\cdots
(R_{01}(z-u_1-\frac{1}{2}hn))^{t_1},
\end{align}
where $L=A_kL_1^{-}(u_1)\cdots
L_k^{-}(u_k)L_k^{+}(u_k+\frac{1}{2}hn)^{-1}\cdots
L_1^{+}(u_1+\frac{1}{2}hn)^{-1}.$ The first crossing symmetry
relation implies
$$(R_{0i}(z-u_i+\frac{1}{2}hn)^{-1})^{t_i}R_{0i}(z-u_i-\frac{1}{2}hn)^{t_i}=id.$$
By using the property
$tr_{1,\cdots,k}~XY=tr_{1,\cdots,k}~X^{t_{1}\cdots
t_{k}}Y^{t_{1}\cdots t_{k}}$, we have
\begin{align}
&tr_{1,\cdots,k}~XY=tr_{1,\cdots,k}~L^{t_{1}\cdots
t_{k}}(R_{01}(z-u_1+\frac{1}{2}hn)^{-1})^{t_1}\cdots
(R_{0k}(z-u_k+\frac{1}{2}hn)^{-1})^{t_k}\nonumber\\
&(R_{0k}(z-u_k-\frac{1}{2}hn))^{t_k}\cdots
(R_{01}(z-u_1-\frac{1}{2}hn))^{t_1}=tr_{1,\cdots,k}~L^{t_{1}\cdots
t_{k}}=tr_{1,\cdots,k}~L
\end{align}
which coincides with $\ell_k(u)$ as defined in (\ref{4}). Thus, $\ell_k(u)$ commutes with $L_0^{+}(z)$. Similarly, we can prove
$L_0^{-}(z)\ell_k(u)=\ell_k(u)L_0^{-}(z)$ by using the second
crossing symmetry relation. Hence, the coefficients of $\ell_k(u)$
belong to the center of the completed Yangian double at the critical
level ${\DY}_h(\mathfrak{gl}_n)_{cr}$.
\end{proof}

\begin{remark}
The theorem is similar to the special case of \cite[Thm. 4.4]{ji:cen} with $\mu=(1^{n})$.
\end{remark}

For the spectral parameter dependent matrix $B(u)=[B(u)_{ij}]$, set
$u_{i}=u+(i-1)h$, the quantum minor $B(u)_{b_{1}\cdots
b_{n}}^{a_{1}\cdots a_{n}}$ is defined by
\begin{align}
A_{n}B_1(u_1)B_2(u_2)\cdots
B_n(u_n)=\sum_{a_{i}b_{i}}e_{a_{1}b_{1}}\otimes e_{a_{2}b_{2}}\cdots
\otimes e_{a_{n}b_{n}}\otimes B(u)_{b_{1}\cdots b_{n}}^{a_{1}\cdots
a_{n}}
\end{align}

The explicit formula for the quantum minors $L^{\pm}(u)_{b_{1}\cdots b_{k}}^{a_{1}\cdots
a_{k}}$ can be written immediately from the definition. For
$a_{1}<\cdots<a_{k}$, it is of Sklyanin determinant type:
\begin{align}\label{e:4.25}
L^{\pm}(u)_{b_{1}\cdots b_{k}}^{a_{1}\cdots
a_{k}}=\sum_{\sigma\in\mathfrak{S}_{k}}sgn(\sigma)l_{a_{\sigma(1)}b_{1}}^{\pm}(u)\cdots
l_{a_{\sigma(k)}b_{k}}^{\pm}(u+(k-1)h)
\end{align}
and then for $\tau\in\mathfrak{S}_{k}$
\begin{align}
L^{\pm}(u)_{b_{1}\cdots b_{k}}^{a_{\tau(1)}\cdots
a_{\tau(k)}}=sgn(\tau)L^{\pm}(u)_{b_{1}\cdots b_{k}}^{a_{1}\cdots
a_{k}}.
\end{align}
Similarly for $b_{1}<\cdots<b_{k}$, we have
\begin{align}
L^{\pm}(u)_{b_{1}\cdots b_{k}}^{a_{1}\cdots
a_{k}}=\sum_{\sigma\in\mathfrak{S}_{k}}sgn(\sigma)l_{a_{k}b_{\sigma(k)}}^{\pm}(u+(k-1)h)\cdots
l_{a_{1}b_{\sigma(1)}}^{\pm}(u)
\end{align}
and then for any $\tau\in\mathfrak{S}_{k}$ we have
\begin{align}\label{e:4.28}
L^{\pm}(u)_{b_{\tau(1)}\cdots b_{\tau(k)}}^{a_{1}\cdots
a_{k}}=sgn(\tau)L^{\pm}(u)_{b_{1}\cdots b_{k}}^{a_{1}\cdots a_{k}}.
\end{align}
It is easy to see that the quantum determinant
$qdetL^{\pm}(u)$ is just the quantum minor $L^{\pm}(u)_{1\cdots n}^{1\cdots n}$.
Moreover, the quantum minor is zero if two top or two bottom indices
are equal.

\begin{theorem}
The coefficients of the quantum determinant
$qdetL^{\pm}(u)$ belong to the
center of the Yangian double $\DY_h(\mathfrak{gl}_n)_{-k}$ at arbitrary level $-k$.
\end{theorem}

\begin{proof}
Introduce the product $R(v_{0},v_{1},\cdots,v_{n})=\prod_{0\leq
a<b\leq n}R_{ab}(v_{a}-v_{b})$, where the $v_{a}$ are variables and
the product is taken in the lexicographical order on the pairs
$(a,b)$. From the defining relations \eqref{e:4.4} and \eqref{e:4.5}, we have
\begin{align}
&R(z+\frac{1}{2}hk,u_{1},\cdots,u_{n})L_{0}^{+}(z)L_1^{-}(u_1)L_2^{-}(u_2)\cdots
L_n^{-}(u_n)\\
&=L_n^{-}(u_n)\cdots
L_2^{-}(u_2)L_1^{-}(u_1)L_{0}^{+}(z)R(z-\frac{1}{2}hk,u_{1},\cdots,u_{n}).\nonumber
\end{align} Since $R(u)=f(u)\bar{R}(u),\prod_{1\leq a<b\leq
n}\bar{R}_{ab}(u_{a}-u_{b})=n!~A_{n}$,and
$$A_{n}L_1^{-}(u_1)L_2^{-}(u_2)\cdots
L_n^{-}(u_n)=L_n^{-}(u_n)\cdots
L_2^{-}(u_2)L_1^{-}(u_1)A_{n}=A_{n}qdetL^{-}(u),$$ by canceling
common factors on both sides of the equality we get
\begin{align}\label{eq 4.30}
&\prod_{a=1}^{n}f(z+\frac{1}{2}hk-u_{a})\prod_{a=1,\cdots,n}^{\rightarrow}\bar{R}_{0a}(z+\frac{1}{2}hk-u_{a})A_{n}L_0^{+}(z)qdetL^{-}(u)\\
&=qdetL^{-}(u)L_0^{+}(z)A_{n}\prod_{a=1,\cdots,n}^{\leftarrow}\bar{R}_{0a}(z-\frac{1}{2}hk-u_{a})\prod_{a=1}^{n}f(z-\frac{1}{2}hk-u_{a}).\nonumber
\end{align}
Observe that
$$\prod_{a=1,\cdots,n}^{\rightarrow}\bar{R}_{0a}(u_{0}-u_{a})A_{n}=A_{n}\prod_{a=1,\cdots,n}^{\leftarrow}\bar{R}_{0a}(u_{0}-u_{a})=A_{n}(1+\frac{h}{u_{0}-u_{1}}).$$
Indeed,by the first equality, it suffices to verify the second
equality on the basis vectors of the form $e_{i}\otimes e_{i}\otimes
e_{1}\otimes \cdots \otimes e_{i-1}\otimes e_{i+1}\otimes \cdots
\otimes e_{n}$ for $i=1,\cdots,n$. The verification is trivial, so we omit it here.
Since $\frac{f(u-nh)}{f(u)}=\frac{u^2-h^2}{u^2}$, we get
\begin{align*}
F(u)=\frac{u^{2}-h^{2}}{u^{2}}F(u),
\end{align*}
where $F(u)=f(u)f(u-h)\cdots f[u-(n-1)h]$. The series $F(u)$ is uniquely determined by this relation. Therefore, we have
\begin{align*}
F(u)=f(u)f(u-h)\cdots f[u-(n-1)h]=(1+hu^{-1})^{-1}.
\end{align*}
It follows that
\begin{align*}
\prod_{a=1}^{n}\frac{f(z-\frac{1}{2}hk-u_{a})}{f(z+\frac{1}{2}hk-u_{a})}=\frac{1+\frac{h}{z+\frac{1}{2}hk-u}}{1+\frac{h}{z-\frac{1}{2}hk-u}}.
\end{align*}
From the above relation and \eqref{eq 4.30}, we can conclude that
$L_{0}^{+}(z)qdetL^{-}(u)=qdetL^{-}(u)L_{0}^{+}(z).$
Using the unitarity
property of the R-matrix $\bar{R}(u)$:
$$\bar{R}(-u)=\frac{u^2-h^2}{u^2}\bar{R}_{21}(u)^{-1},$$
the relation $L_{0}^{-}(z)qdetL^{+}(u)=qdetL^{+}(u)L_{0}^{-}(z)$ can be checked
by a similar argument.
The remaining two relations
$L_{0}^{\pm}(z)qdetL^{\pm}(u)=qdetL^{\pm}(u)L_{0}^{\pm}(z)$ are similarly obtained by using the defining relations \eqref{e:4.4}.
\end{proof}

\begin{remark}
The proof of Theorem 4.5 is a little bit different from those in \cite[Prop. 2.8]{ji:cen}.
\end{remark}

\begin{theorem}
The coefficients
of $qdetL^{\pm}(u)$ are algebraically independent and generate the
center of $\DY_{h}(\mathfrak{gl}_n)_{cr}$.
\end{theorem}

\begin{proof}
Let $qdetL^{+}(u)=1+h\sum_{r\geq
1}d_{r}^{+}u^{-r}$, $qdetL^{-}(u)=1-h\sum_{r\geq 1}$\\
$d_{r}^{-}u^{r-1}$. Introduce the filtration on the Yangian double at the critical level
$\DY(\mathfrak{gl}_n)_{cr}$ by ${\rm deg}\,l_{ij}^{r}=r-1,{\rm deg}\,
l_{ij}^{-r}=-r$ for all $r\geq 1$ and ${\rm deg}\,h=0$. Consider the central extension
$\hat{\mathfrak{gl}_n}=\mathfrak{gl}_n[t,t^{-1}]\oplus \mathbb{C}K$
defined by the commutation relations
\begin{align*}
[E_{ij}[r],E_{kl}[s]]=\delta_{kj}E_{il}[r+s]-\delta_{il}E_{kj}[r+s]+r\delta_{r,-s}K(\delta_{kj}\delta_{il}-\frac{\delta_{ij}\delta_{kl}}{n}).
\end{align*}
As in \cite[Prop. 2.1]{ji:cen},
the corresponding graded algebra ${\rm gr}\,\DY(\mathfrak{gl}_n)_{cr}$
is isomorphic to
$U(\hat{\mathfrak{gl}_n})_{cr}[[h]]$, where
$U(\hat{\mathfrak{gl}_n})_{cr}$ is the quotient of $U(\hat{\mathfrak{gl}_n})$ modulo the ideal generated by the relation $K=-n$.
Now we derived from the definition of quantum determinant that the
coefficient $d_{r}^{+}$ of $qdetL^{+}(u)$ has the form
$d_{r}^{+}=l_{11}^{r}+\cdots+l_{nn}^{r}$ plus terms of degree less
than $r-1$. Therefore, the image of $d_{r}^{+}$ in the $(r-1)$-th
component of ${\rm gr}\,\DY(\mathfrak{gl}_n)_{cr}$ coincides with $It^{r-1}$
where $I=E_{11}+\cdots+E_{nn}$. Similarly, the image of $d_{r}^{-}$
in the $(-r)$-th component of ${\rm gr}\,\DY(\mathfrak{gl}_n)_{cr}$
coincides with $It^{-r}$. The elements $It^{r}(r\in\mathbb{Z})$
are algebraically independent, so are the elements
$d_{r}^{+},d_{r}^{-}$. The elements $It^{r}(r\in\mathbb{Z})$
generate the center of $U(\hat{\mathfrak{gl}_n})_{cr}$. This implies
that $d_{1}^{+},d_{2}^{+},\cdots,d_{1}^{-},d_{2}^{-},\cdots$
generate the center of $\DY_{h}(\mathfrak{gl}_n)_{cr}$.
\end{proof}

As a corollary of Theorem 4.4, we describe the
invariants of the vacuum module over the Yangian double explicitly. First of all, we define the vacuum module at the critical level
$V_{h}(\mathfrak{gl}_n)$ to be the quotient of
$\DY_h(\mathfrak{gl}_n)_{cr}$ by the left ideal generated by all the
elements $l_{ij}^{r}$ with $r\geq 0$. The module
$V_{h}(\mathfrak{gl}_n)$ is generated by the vector $\textbf{1}$
(the image of $1\in \DY_h(\mathfrak{gl}_n)_{cr}$ in the quotient)
such that
\begin{align*}
L^{+}(z)\textbf{1}=I\textbf{1},
\end{align*}
where $I$ is the identity matrix. The subspace of invariants of
$V_{h}(\mathfrak{gl}_n)$ is defined by
\begin{align*}
\mathfrak{Z}_{h}(\mathfrak{gl}_n)=\{v\in V_{h}(\mathfrak{gl}_n)\mid
L^{+}(z)v=Iv\}.
\end{align*}

For $k=1,\cdots,n$, introduce the series $\bar{\ell}_{k}(u)$ with
coefficients in $V_{h}(\mathfrak{gl}_n)$ by
\begin{align*}
\bar{\ell}_{k}(u)=tr_{1,\cdots,k}~A_kL_1^{-}(u_1)\cdots
L_k^{-}(u_k).
\end{align*}

\begin{corollary}
All coefficients of the series $\bar{\ell}_{k}(u)\textbf{1}$ with
$k=1,\cdots,n$ belong to the subspace of invariants
$\mathfrak{Z}_{h}(\mathfrak{gl}_n)$. Moreover, the coefficients of
all series $\bar{\ell}_{k}(u)$ commute with each other.
\end{corollary}

\begin{proof}
By Theorem 4.4, we have $L^{+}(z)\ell_k(u)=\ell_k(u)L^{+}(z)$.
Note that $\ell_k(u)\textbf{1}=\bar{\ell}_k(u)\textbf{1}$,
so the first part of the corollary follows by the application of both sides of $L^{+}(z)\ell_k(u)=\ell_k(u)L^{+}(z)$ to the vector $\textbf{1}\in V_{h}(\mathfrak{gl}_n)$.
We can prove the second part by applying both sides of the equality
$\ell_k(u)\ell_m(v)=\ell_m(v)\ell_k(u)$ to the vector $\textbf{1}$.
For the left hand side, we get
\begin{align*}
\ell_k(u)\ell_m(v)\textbf{1}=\ell_k(u)\bar{\ell}_m(v)\textbf{1}=\bar{\ell}_m(v)\ell_k(u)\textbf{1}=\bar{\ell}_m(v)\bar{\ell}_k(u)\textbf{1}.
\end{align*}

Similarly, for the right hand side we have
\begin{align*}
\ell_m(v)\ell_k(u)\textbf{1}=\bar{\ell}_k(u)\bar{\ell}_m(v)\textbf{1}.
\end{align*}
Therefore, $\bar{\ell}_k(u)\bar{\ell}_m(v)=\bar{\ell}_m(v)\bar{\ell}_k(u)$.
\end{proof}

\begin{remark}
The series $\bar{\ell}_{k}(u)$ coincides with $\mathbb{T}^{+}_{\mu}(u)$ with $\mu=(1^{n})$, see \cite[Cor. 4.6]{ji:cen}.
\end{remark}

By calculating the trace in (4.7), we can write $\ell_k(u)$ in terms of the entries of
$L^{-}(u)=[l_{ij}^{-}(u)]$ and
$\tilde{L}^{+}(u):=L^{+}(u)^{-1}=[\tilde{l}_{ij}^{+}(u)]$.

\begin{proposition}
For $k=1,\cdots,n$, we have
\begin{multline}\label{9}
\ell_k(u)=\sum_{j_{1},\cdots,j_{k}}\sum_{i_{1}<i_{2}<\cdots<i_{k}}\sum_{\sigma\in\mathfrak{S}_{k}}sgn(\sigma)l_{i_{\sigma(1)}j_{1}}^{-}(u)\cdots
l_{i_{\sigma(k)}j_{k}}^{-}(u+(k-1)h)\\
\tilde{l}_{j_{k}i_{k}}^{+}(u+(k-1)h+\frac{1}{2}hn)\cdots
\tilde{l}_{j_{1}i_{1}}^{+}(u+\frac{1}{2}hn)
\end{multline}
and
\begin{multline}\label{10}
\ell_k(u)=\sum_{j_{1},\cdots,j_{k}}\sum_{i_{1}<i_{2}<\cdots<i_{k}}\sum_{\sigma\in\mathfrak{S}_{k}}sgn(\sigma)l_{i_{\sigma(k)}j_{k}}^{-}(u)\cdots
l_{i_{\sigma(1)}j_{1}}^{-}(u+(k-1)h)\\
\tilde{l}_{j_{1}i_{1}}^{+}(u+(k-1)h+\frac{1}{2}hn)\cdots
\tilde{l}_{j_{k}i_{k}}^{+}(u+\frac{1}{2}hn).
\end{multline}
\end{proposition}

\begin{proof}
Using (\ref{4}), we regard
$$A_kL_1^{-}(u_1)\cdots
L_k^{-}(u_k)\tilde{L}_k^{+}(u_k+\frac{1}{2}hn)\cdots
\tilde{L}_1^{+}(u_1+\frac{1}{2}hn)
$$
with $u_{i}=u+(i-1)h$ as an operator in the vector space
$(\mathbb{C}^n)^{\otimes k}$. Since the operator is divisible on the right by
$A_k$, its trace equals to $k!$ times the sum of the
diagonal matrix elements corresponding to basis vectors of the form
$e_{i_{1}}\otimes\cdots\otimes e_{i_{k}}$ with $i_{1}<\cdots<i_{k}$.
Note that $sgn(\sigma)=sgn(\sigma^{-1})$, and for any $\sigma\in\mathfrak{S}_{k}$,
\begin{align*}
\sigma(e_{i_{1}}\otimes\cdots\otimes
e_{i_{k}})=e_{i_{\sigma(1)}}\otimes\cdots\otimes e_{i_{\sigma(k)}},
\end{align*}
(\ref{9}) immediately follows from (\ref{8}). Similarly, we can prove
(\ref{10}) by using the basis vectors
$e_{i_{k}}\otimes\cdots\otimes e_{i_{1}}$ with the same condition
$i_{1}<\cdots<i_{k}$ on the indices.
\end{proof}

We have the following well-known expression for $\ell_n(u)$ (cf. \cite{ji:cen}).
\begin{corollary}
\begin{align}
\ell_n(u)=qdetL^{-}(u)(qdetL^{+}(u+\frac{1}{2}hn))^{-1}.
\end{align}
\end{corollary}

\begin{proof}
The definition of the quantum determinant implies
\begin{align}
A_{n}L_{n}^{+}(u_{n})^{-1}\cdots
L_{1}^{+}(u_{1})^{-1}=(qdetL^{+}(u))^{-1}A_{n}.
\end{align}
Replacing $u$ by $u+\frac{1}{2}hn$ and using
$tr_{1,\cdots,n}A_{n}=1$ we obtain the desired formula (4.33) from
(\ref{4}).
\end{proof}

\begin{lemma}
The entries of the inverse matrix $L^{+}(u)^{-1}$ are given by
\begin{align}
[L^{+}(u)^{-1}]_{ij}=(-1)^{j-i}(qdetL^{+}(u-(n-1)h))^{-1}L^{+}(u-(n-1)h)_{1\cdots
\hat{i}\cdots n}^{1\cdots \hat{j}\cdots n},
\end{align}
where the hat means omitting the indices.
\end{lemma}

\begin{proof}
By the definition of the quantum determinant, we have
\begin{align*}
A_{n}L_{1}^{+}(u_{1})\cdots
L_{n-1}^{+}(u_{n-1})=A_{n}qdetL^{+}(u)L_{n}^{+}(u_{n})^{-1},
\end{align*}
where $u_{i}=u+(i-1)h$. Now apply both sides to the basis vector
$e_{1}\otimes\cdots\otimes\hat{e_{i}}\otimes\cdots\otimes
e_{n}\otimes e_{j}$ and replace $u$ by $u-(n-1)h$, we get the desired formula (4.35).
\end{proof}

\newpage
\section{Harish-Chandra homomorphisms}\label{s:HS-hom}

We discuss the Poincar\'{e}-Birkhoff-Witt theorem for the Yangian
double. Define a total ordering $\prec$ on the set of generators
as follows. First, each generator $l_{ij}^{(r)}(r<0)$ precedes any
generator $l_{km}^{(s)}(s\geq 0)$. Furthermore,$l_{ij}^{(r)}\prec
l_{km}^{(s)}(r,s<0)$ if and only if the triple $(j-i,i,r)$ precedes
$(m-k,k,s)$ in the lexicographical order. Finally, we set
$l_{ij}^{(r)}\prec l_{km}^{(s)}(r,s\geq 0)$ if and only if the triple
$(i-j,i,r)$ precedes $(k-m,k,s)$ in the lexicographical
order. Consider the Yangian double $\DY_h(\mathfrak{gl}_n)_{cr}$ at
the critical level $c=-n$. By PBW theorem for the Yangian double, any
element $x\in \DY_h(\mathfrak{gl}_n)_{cr}$ can be written as a
unique linear combination of ordered monomials in the generators
$l_{ij}^{(k)}$. Let $DY^{0}$ be the subspace of $\DY_h(\mathfrak{gl}_n)_{cr}$
spanned by those monomials which do not contain any generators
$l_{ij}^{(k)}$ with $i\neq j$. Denote by $\theta$ the projection from $\DY_h(\mathfrak{gl}_n)_{cr}$ to
$DY^{0}$, and let $x_{0}=\theta(x)$. Then we extend $\theta$ by continuity to get the projection
$\theta:{\DY}_h(\mathfrak{gl}_n)_{cr}\rightarrow \widetilde{DY}^{0}$, where $\widetilde{DY}^{0}$ is
the corresponding completed vector space of $DY^{0}$.

Let $\Pi_{c}(n)$ be the algebra of polynomials in
independent variables $l_{i}^{k}$ with $i=1,\cdots,n$ and $k\in
\mathbb{Z}$.
The mapping $l_{ii}^{(k)}\mapsto l_{i}^{k}$ extends to an isomorphism of vector
spaces $\eta:DY^{0}\rightarrow \Pi_{c}(n).$
Define the completion $\widetilde{\Pi}_{c}(n)$ as the inverse limit
$$\widetilde{\Pi}_{c}(n)=\underleftarrow{lim}\Pi_{c}(n)/I_{p},p>0$$
where $I_{p}$ is the ideal of $\Pi_{c}(n)$ generated by all
elements $l_{i}^{k}$ with $k\geq p$. We can extend $\eta$
to an isomorphism of the respective completed vector spaces
$\eta:\widetilde{DY}^{0}\rightarrow \widetilde{\Pi}_{c}(n)$. Let $\chi=\eta\circ\theta$, then $\chi$ is
a linear map $\chi:{\DY}_h(\mathfrak{gl}_n)_{cr}\rightarrow
\widetilde{\Pi}_{c}(n)$. In the following proposition, we give an analogue of
the Harish-Chandra homomorphism for the Yangian double at the critical level.

\begin{proposition}
The restriction of the map $\chi$ to the center $Z_{c}(\mathfrak{gl}_n)$ of
the algebra ${\DY}_h(\mathfrak{gl}_n)_{cr}$ is a homomorphism of
commutative algebras.
\end{proposition}

\begin{proof}
For $x,y\in Z_{c}(\mathfrak{gl}_n)$,set $x_0=\chi(x),y_0=\chi(y)$.
By the PBW theorem for the Yangian double, we can write $y$ as a
(possibly infinite) linear combination of ordered monomials in the
generators $l_{ij}^{(r)}$. Suppose that
$$m=\prod_{a} l_{i_a j_a}^{(r_a)}\prod_{b} l_{i_b j_b}^{(r_b)} \quad(r_a<0,r_b\geq0)$$
is an ordered monomial which occurs in the linear combination.From
the defining relations (3.4) and (3.5),we get
\begin{align}
[l_{ij}^{(0)},l_{km}^{\pm}(v)]=\delta_{kj}l_{km}^{\pm}(v)-\delta_{im}l_{ki}^{\pm}(v).
\end{align}
Therefore,
\begin{equation}\label{e:5.2}
{[l_{ii}^{(0)},l_{km}^{\pm}(v)]=}\begin{cases}
l_{km}^{\pm}(v)& k=i,m\neq i\\ 
-l_{km}^{\pm}(v)& k\neq i,m=i\\ 
0&\mbox{otherwise}
\end{cases}
\end{equation}
Since $\displaystyle[l_{ii}^{(0)},m]=[l_{ii}^{(0)},\prod_{a} l_{i_a
j_a}^{(r_a)}\prod_{b} l_{i_b j_b}^{(r_b)}]=0$ for $i=1,\cdots,n$, we
have
\begin{align}\label{e:5.5}
\sum_{a}(i_a-j_a)+\sum_{b}(i_b-j_b)=0,
\end{align}
which is implied by \eqref{e:5.2}. 
Suppose that $m\in
Ker\chi$. From \eqref{e:5.5}, we have $i_a>j_a$ for some $a$ or $i_b>j_b$ for
some $b$.Since $x$ is in the center,we have
$$xm=\prod_{a} l_{i_a j_a}^{(r_a)}x\prod_{b} l_{i_b j_b}^{(r_b)}.$$
We can write $xm$ as a linear combination of ordered monomials by using
the defining relations (3.4). Due to the property
that $i_a>j_a$ for some $a$ or $i_b>j_b$ for some $b$, we derive that
$xm\in Ker\chi$.Thus only $\chi(xy_0)$ can give a nonzero contribution to the image $\chi(xy)$, and $xy_0$ has an expression of the
form $$\prod_{a} l_{i_a i_a}^{(r_a)}x\prod_{b} l_{i_b i_b}^{(r_b)}.$$ If
$p$ is an ordered monomial which occurs in the linear combination
representing $x$ and $\chi(p)=0$,then we conclude that
$\displaystyle\chi(\prod_{a} l_{i_a i_a}^{(r_a)}p\prod_{b} l_{i_b
i_b}^{(r_b)})=0$ by applying property \eqref{e:5.5} to the monomial
$p$. Finally, by (3.4) we can swap any two
generators $l_{ii}^{(r)}$ and $l_{jj}^{(s)}$ with $r,s\geq 0$
(resp.$r,s<0$) modulo $Ker\chi$ within any monomial
of the form $\displaystyle\prod_{a} l_{i_a i_a}^{(r_a)}\prod_{b}
l_{i_b i_b}^{(r_b)}$. This implies that
$\chi(xy)=x_{0}y_{0}=\chi(x)\chi(y)$.
\end{proof}

Now we will find the image of the Harish-Chandra map for
$\ell_k(u)$ defined in Theorem 4.4. Introduce the following series
whose coefficients are the generators of the algebra $\Pi_{c}(n)$
$$l_{i}^{+}(u)=1-h\sum_{k\in\mathbb{Z}_{\geq 0}}l_{i}^{k}u^{-k-1},\quad l_{i}^{-}(u)=1+h\sum_{k\in\mathbb{Z}_{<
0}}l_{i}^{k}u^{-k-1}$$ and for $i=1,\cdots,n$ set
$$\lambda_{i}(u)=\frac{l_{i}^{-}(u)l_{1}^{+}(u-h+\frac{1}{2}hn)\cdots l_{i-1}^{+}(u-(i-1)h+\frac{1}{2}hn)}{l_{1}^{+}(u+\frac{1}{2}hn)\cdots l_{i}^{+}(u-(i-1)h+\frac{1}{2}hn)}.$$
Note that $\lambda_{i}(u)$ a Laurent series in $u$ with
coefficients in the completed algebra
$\widetilde{\Pi}_{c}(n)$.

\begin{theorem}
For each $k=1,\cdots,n$ the image of the series $\ell_k(u)$ under
the Harish-Chandra homomorphism is given by
$$\chi:\ell_k(u)\mapsto \sum_{1\leq i_{1}<\cdots<i_{k}\leq n}\lambda_{i_1}(u)\lambda_{i_2}(u+h)\cdots\lambda_{i_k}(u+(k-1)h).$$
\end{theorem}

\begin{proof}
We can express $[L^{+}(u)^{-1}]_{ij}$ in terms of quantum minors by using Lemma 4.9.
From the definition of the quantum determinant,we have
$$L^{+}(u)_{1\cdots\hat{i}\cdots n}^{1\cdots\hat{j}\cdots n}=sgn(\omega)L^{+}(u)_{n\cdots\hat{i}\cdots 1}^{n\cdots\hat{j}\cdots 1},$$
where $\omega\in \mathfrak{S}_{n-1}$ reverses the order of the lower
indices. Now expand this quantum minor by \eqref{e:4.25} and use formula (4.31) for $\ell_k(u)$, we see that a
nonzero contribution to the image $\chi(\ell_k(u))$ can only come
from the summands in (4.31) with $i_{\sigma(1)}\leq j_1\leq
i_1$. It follows that $\sigma(1)=1$ and $i_1=j_1$. Similarly, the
defining relations of $\DY_{h}(\mathfrak{gl}_n)$ imply that
$\sigma(2)=2$ and $i_2=j_2$,etc. Thus a nonzero
contribution only comes from the summands with $\sigma=1$ and $i_a=j_a$
for all $a=1,\cdots,k$. After apply Lemma 4.9 and formulas \eqref{e:4.25}
again, we conclude that the images of the quantum minors are
given by
$$qdetL^{+}(u)\mapsto l_{1}^{+}(u+(n-1)h)\cdots l_{n}^{+}(u)$$ and
$$L^{+}(u)_{1\cdots\hat{i}\cdots n}^{1\cdots\hat{i}\cdots n}\mapsto l_{1}^{+}(u+(n-2)h)\cdots l_{i-1}^{+}(u+(n-i)h)l_{i+1}^{+}(u+(n-i-1)h)\cdots l_{n}^{+}(u).$$
This finishes the calculation of the Harish-Chandra image of
$\ell_k(u)$.
\end{proof}

\begin{remark}
Note that Molev and Mukhin's paper \cite{MM} constructed an analogue of the Harish-Chandra homomorphism for the Yangian $Y(\mathfrak{gl}_n)$. Here we extend this result to the
Yangian double at the critical level ${\DY}_h(\mathfrak{gl}_n)_{cr}$ using a different method.
\end{remark}

\newpage
\section{Eigenvalues in Wakimoto modules}\label{s:eigenvalue}
We recall the construction of $h$-deformed Wakimoto modules over $\DY_{h}(sl_n)$ at the critical level. It can be
realized in the bosonic Fock space by an explicit action of the generator series
$H_{i}^{\pm}(u),E_{i}(u),F_{i}(u)$ as in \cite{HZ}. We reproduce this construction as follows.\\

Introduce the following set of $n^2-1$ Heisenberg algebras with
generators $a_{n}^{i}(1\leq i\leq n-1),b_{n}^{ij}$ and
$c_{n}^{ij}(1\leq i<j\leq n)$ with $n\in \mathbb{Z}-\{0\}$ and
$p_{a^i},q_{a^i}(1\leq i\leq n-1)$,$p_{b^{ij}},
q_{b^{ij}},p_{c^{ij}},q_{c^{ij}}(1\leq i<j\leq n)$.
\begin{align}\label{e:6.1}
[a_{n}^{i},a_{m}^{j}]&=(k+g)B_{ij}n\delta_{n+m,0},\quad
[p_{a^i},q_{a^j}]=(k+g)B_{ij},\\ \label{e:6.2}
[b_{n}^{ij},b_{m}^{i^{\prime}j^{\prime}}]&=-n\delta_{i,i^{\prime}}\delta_{j,j^{\prime}}\delta_{n+m,0},\quad
[p_{b^{ij}},q_{b^{i^{\prime}j^{\prime}}}]=-\delta_{i,i^{\prime}}\delta_{j,j^{\prime}}\\ \label{e:6.3}
[c_{n}^{ij},c_{m}^{i^{\prime}j^{\prime}}]&=n\delta_{i,i^{\prime}}\delta_{j,j^{\prime}}\delta_{n+m,0},\quad
[p_{c^{ij}},q_{c^{i^{\prime}j^{\prime}}}]=\delta_{i,i^{\prime}}\delta_{j,j^{\prime}}
\end{align}
where $g=n$ is the dual Coxeter number for the Cartan matrix of type
$A_{n-1}$,$B_{ij}=\frac{1}{2}a_{ij}$, where $a_{ij}$ are entries of the
Cartan matrix of the type $A_{n-1}$.\\

The Fock space $F_{h}(n)$ corresponding to the above Heisenberg
algebras can be defined as follows. Let $\mid0\rangle$ be the
vacuum state defined by
\begin{align*}
a_{n}^{i}\mid0\rangle &=b_{n}^{ij}\mid0\rangle=c_{n}^{ij}\mid0\rangle=0\quad(n>0),\\
p_{a^i}\mid0\rangle &=p_{b^{ij}}\mid0\rangle=p_{c^{ij}}\mid0\rangle=0.
\end{align*}

For $X=a^{i},b^{ij},c^{ij}$, let us now define
\begin{align*}
&X(u;A,B)=\sum_{n>0}\frac{X_{-n}}{n}(u+Ah)^{n}-\sum_{n>0}\frac{X_{n}}{n}(u+Bh)^{-n}+log(u+Bh)p_{X}+q_{X},\\
&X_{+}(u;B)=-\sum_{n>0}\frac{X_{n}}{n}(u+Bh)^{-n}+log(u+Bh)p_{X},\\
&X_{-}(u;A)=\sum_{n>0}\frac{X_{-n}}{n}(u+Ah)^{n}+q_{X},\\
&X(u;A)=X(u;A,A),X(u)=X(u;0).
\end{align*}

For $b^{ij},c^{ij}$, define
\begin{align*}
\hat{X}_{\pm}(u)=\mp(X_{\pm}(u;-\frac{1}{2})-X_{\pm}(u;\frac{1}{2})).
\end{align*}

For the bosonic fields $a^{i}(u;A,B)$, define
\begin{align*}
&\hat{a}_{+}^{i}(u)=a_{+}^{i}(u;0)-a_{+}^{i}(u;k+g),\\
&\hat{a}_{-}^{i}(u)=\frac{1}{k+g}\sum_{j,l=1}^{n-1}(B^{-1})^{jl}(a_{-}^{j}(u;B_{ij})-a_{-}^{j}(u;-B_{ij})).
\end{align*}

The ($h$-deformed) Wakimoto module over $\DY_{h}(sl_n)$ at the
critical level is realized by defining the action of the series
$H_{i}^{\pm}(u),E_{i}(u),F_{i}(u)$ in the Fock space $F_{h}(n)$.
Here we only give the formulas for the action of $H_{i}^{\pm}(u)$ and $E_{i}(u)$.
\begin{align}\nonumber
H_{i}^{\pm}(u)&=:exp\{\sum_{l=1}^{i}\hat{b}_{\pm}^{l,i+1}(u\pm
\frac{1}{2}(\frac{k}{2}+l-1)h)\\ \label{e:6.4}
& -\sum_{l=1}^{i-1}\hat{b}_{\pm}^{li}(u\pm
\frac{1}{2}(\frac{k}{2}+l)h)+\hat{a}_{\pm}^{i}(u\mp\frac{1}{4}kh)\\ \nonumber
&+\sum_{l=i+1}^{n}\hat{b}_{\pm}^{il}(u\pm
\frac{1}{2}(\frac{k}{2}+l)h)-\sum_{l=l+2}^{n}\hat{b}_{\pm}^{i+1,l}(u\pm
\frac{1}{2}(\frac{k}{2}+l-1)h)\}:,
\end{align}
and
\begin{align}\label{e:6.5}
&E_{i}(u)=-\frac{1}{h}\sum_{m=1}^{i}:exp\{(b+c)^{mi}(u+\frac{1}{2}(m-1)h)\}\\\nonumber
&\times[exp(\hat{b}_{+}^{m,i+1}(u+\frac{1}{2}(m-1)h)-(b+c)^{m,i+1}(u+\frac{1}{2}mh))\\\nonumber
&-exp(\hat{b}_{-}^{m,i+1}(u+\frac{1}{2}(m-1)h)-(b+c)^{m,i+1}(u+\frac{1}{2}(m-2)h))]\\\nonumber
&\times
exp\{\sum_{l=1}^{m-1}[\hat{b}_{+}^{l,i+1}(u+\frac{1}{2}(l-1)h)-\hat{b}_{+}^{li}(u+\frac{1}{2}lh)]\}:
\end{align}
where $(b+c)^{ij}(u)=b^{ij}(u)+c^{ij}(u)$. We define a normal ordering in the product so that the coefficients $b_{n}^{ij}$ with $n<0$ or
$q_{b^{ij}}$ should be placed to the left of the coefficients
$b_{n}^{ij}$ with $n>0$ or $p_{b^{ij}}$. A similar rule applies to the coefficients of $c^{ij}(u)$.
By Theorem 4.5, the coefficients of the quantum determinants are
central elements in the algebra $\DY_{h}(\mathfrak{gl}_n)_{cr}$. Therefore, we can extend the irreducible Wakimoto modules
to $\DY_h(\mathfrak{gl}_n)_{cr}$ by specifying the
eigenvalues $K^{\pm}(u)$ of $qdetL^{\pm}(u)$. Lemma 2.4 implies the following conditions
\begin{align*}
k_{1}^{\pm}(u)k_{2}^{\pm}(u+h)\cdots k_{n}^{\pm}(u+(n-1)h)\mapsto
K^{\pm}(u),
\end{align*}
where $K^{-}(u)$ and $K^{+}(u)$ are power series in $u$ and
$u^{-1}$, respectively. Hence, we can use relation \eqref{e:6.4} to define the
action of the coefficients of all series $k_{i}^{\pm}(u)$ in the
space $F_{h}(n)$.
For any $X\in \DY_h(\mathfrak{gl}_n)_{cr}$ we will denote
$\langle0\mid X\mid0 \rangle$ by the coefficient of $\mid0
\rangle$ in the expansion of $X\mid0 \rangle$. Moreover, a relation of the form $\langle0\mid
X=\langle0\mid d$ for a constant $d$ means that $\langle0\mid XY\mid0 \rangle=d\langle0\mid Y\mid0 \rangle$ for
any element $Y\in \DY_h(\mathfrak{gl}_n)_{cr}$. Next we will use this notation
to parameterize the corresponding modules over
$\DY_h(\mathfrak{gl}_n)_{cr}$ by the power series
$\varkappa_{i}^{+}(u)$ and $\varkappa_{i}^{-}(u)$ in $u^{-1}$ and
$u$, respectively, such that
\begin{align}\label{e:6.6}
k_{i}^{+}(u)\mid0 \rangle=\varkappa_{i}^{+}(u)\mid0 \rangle \quad
and \quad k_{i}^{-}(u)\mid0 \rangle=\varkappa_{i}^{-}(u)\mid0
\rangle
\end{align}
for all $i=1,\cdots,n$. The series $\varkappa_{i}^{\pm}(u)$ satisfy the following relations
\begin{align*}
&\varkappa_{i+1}^{+}(u)\varkappa_{i}^{+}(u)^{-1}=exp(\hat{a}_{+}^{i}(u+\frac{n-2i}{4}h)),\\
&\varkappa_{i+1}^{-}(u)\varkappa_{i}^{-}(u)^{-1}=exp(\hat{a}_{-}^{i}(u-\frac{n+2i}{4}h))
\end{align*}
for all $i=1,\cdots,n-1$. Since $\hat{a}_{+}^{i}(u)=0$, the series
$\varkappa_{i}^{+}(u)$ is the same for each $i$. We will denote them by
$\varkappa^{+}(u)$ for convenience.

\begin{theorem}
Given an irreducible Wakimoto module over
$\DY_h(\mathfrak{gl}_n)_{cr}$ with the parameters $\varkappa^{+}(u)$
and $\varkappa_{i}^{-}(u)$, we have the following formulas for the eigenvalues of
the series $\ell_k(u)$ in the module
\begin{align*}
\ell_k(u)\mapsto \sum_{1\leq i_{1}<\cdots<i_{k}\leq
n}\Lambda_{i_1}(u)\Lambda_{i_2}(u+h)\cdots\Lambda_{i_k}(u+(k-1)h),\quad
k=1,\cdots,n,
\end{align*}
where
\begin{align*}
\Lambda_{i}(u)=\varkappa_{i}^{-}(u)\varkappa^{+}(u+\frac{1}{2}hn)^{-1},\quad
i=1,\cdots,n.
\end{align*}
\end{theorem}

\begin{proof}
Any irreducible Wakimoto module is generated by the vacuum vector
$\mid0\rangle$ over $\DY_h(\mathfrak{gl}_n)_{cr}$. So the eigenvalues of the
series $\ell_k(u)$ can be found by calculating the series
$\langle0\mid\ell_{k}(u)\mid0
\rangle$.
It follows from \eqref{e:6.5} that $E_{i}(u)\mid0\rangle$ is a power series in
$u$ for all $i$. Therefore, the definition of $E_{i}(u)$ implies
$e_{i,i+1}^{+}(u)\mid0\rangle=0$ for $i=1,\cdots,n-1$. We have the
following relations for the action of the generators $l_{ij}^{(0)}$ on
the vacuum vector $\mid0\rangle$:
\begin{align}\label{e:6.7}
l_{i+1,i}^{(0)}\mid0\rangle=0\quad and \quad
l_{ii}^{(0)}\mid0\rangle=-\frac{1}{h}(c_{i}-c_{i-1})\mid0\rangle,
\end{align}
where $c_{0}=0$ and $c_{i}(i\geq 1)$ denotes the coefficient of $u^{-1}$
in $\varkappa^{+}(u)\cdots\varkappa^{+}(u+(i-1)h)$. Indeed,
\eqref{e:2.10} and \eqref{e:6.6} imply that $L^{+}(u)_{1\cdots i}^{1\cdots
i}\mid0\rangle$ is a scalar power series in $u^{-1}$. Using \eqref{e:4.25} to expand the
quantum minor and taking the coefficient of $u^{-1}$ we
get that $(l_{11}^{(0)}+\cdots+l_{ii}^{(0)})\mid0\rangle$ is a scalar
multiple of the vacuum vector $\mid0\rangle$. Indeed, we have
$L^{+}(u)_{1\cdots i}^{1\cdots i}=k_{1}^{+}(u)\cdots
k_{i}^{+}(u+(i-1)h)$. Therefore, $-h(l_{11}^{(0)}+\cdots+l_{ii}^{(0)})\mid0\rangle=c_{i}\mid0\rangle$, which
implies the second relation in \eqref{e:6.7}. Now use the relation
$L^{+}(u)_{1\cdots i-1,i+1}^{1\cdots i-1,i}\mid0\rangle=0$ implied
by (3.7). Taking the coefficient of $u^{-1}$, we get
$l_{i+1,i}^{(0)}\mid0\rangle=0$. As a next step, we will prove
\begin{align}\label{e:6.8}
e_{i,j}^{+}(u)\mid0\rangle=0\quad for\ all\quad i<j.
\end{align}
(5.1) implies that
\begin{align}
[l_{j+1,j}^{(0)},l_{km}^{\pm}(u)]=0\quad for\ all\quad j>k,m
\end{align}
and
\begin{align}
[l_{j+1,j}^{(0)},l_{ji}^{\pm}(u)]=l_{j+1,i}^{\pm}(u)\quad for\
all\quad j>i.
\end{align}
Hence, (3.7) implies
\begin{align}
[l_{j+1,j}^{(0)},e_{i,j}^{+}(u)]=e_{i,j+1}^{+}(u)\quad for\ all\quad
j>i
\end{align}
and we get \eqref{e:6.8} from \eqref{e:6.7} by induction.
Therefore, we have
\begin{align}
L^{+}(u)_{1\cdots i-1,i}^{1\cdots i-1,j}\mid0\rangle=0\quad for\
all\quad i<j.
\end{align}
Indeed, use \eqref{e:4.25} and \eqref{e:4.28} to expand the quantum minor we get
\begin{align}\label{e:6.13}
\sum_{\sigma}sgn(\sigma) l_{\sigma(j),i}^{+}(u)\cdots
l_{\sigma(1),1}^{+}(u+(i-1)h)\mid0\rangle=0,
\end{align}
where the sum is taken over the permutations $\sigma$ of the set
$\{1,\cdots,i-1,j\}$. Since
\begin{align}\label{e:6.14}
L^{+}(u)_{1\cdots i}^{1\cdots
i}\mid0\rangle=\sum_{\sigma}sgn(\sigma) l_{\sigma(i),i}^{+}(u)\cdots
l_{\sigma(1),1}^{+}(u+(i-1)h)\mid0\rangle
\end{align}
is a scalar power series in $u^{-1}$, it follows that
\begin{align}\label{e:6.15}
l_{ji}^{+}(u)\mid0\rangle=0\quad for\ all\quad j>i
\end{align}
and $l_{ii}^{+}(u)\mid0\rangle$ is a scalar power series in $u^{-1}$
by induction on $j-i$. Furthermore, (3.7) and
formulas \eqref{e:6.6} imply that
\begin{align}\label{e:6.16}
l_{ii}^{+}(u)\mid0\rangle=\varkappa^{+}(u)\mid0\rangle
\end{align}
for all $i=1,\cdots,n$. Similarly, we have
\begin{align}\label{e:6.17}
\langle0\mid l_{ji}^{-}(u)=0\quad for\ all\ j>i\quad and
\quad\langle0\mid l_{ii}^{-}(u)=\langle0\mid
\varkappa_{i}^{-}(u)\quad for
\end{align}
for $i=1,\cdots, n$.

We observe that $\langle0\mid E_{i}(u)$ is a power
series in $u^{-1}$. Indeed, this fact follows from \eqref{e:6.5} by using
the relations
\begin{align}\label{e:6.18}
\langle0\mid a_{n}^{i}=\langle0\mid b_{n}^{ij}=\langle0\mid
c_{n}^{ij}=0\quad for\ all\quad n<0.
\end{align}
An extra step is to apply the commutation relations \eqref{e:6.2} and
\eqref{e:6.3} which implies that
\begin{align*}
exp(q_{b^{ij}})\cdot u^{p_{b^{ij}}}=u^{p_{b^{ij}}}\cdot exp
(q_{b^{ij}})\cdot u
\end{align*}
and
\begin{align*}
exp(q_{c^{ij}})\cdot u^{p_{c^{ij}}}=u^{p_{c^{ij}}}\cdot exp
(q_{c^{ij}})\cdot u^{-1}.
\end{align*}
Hence, these relations and the
definition of $E_{i}(u)$ imply the relation $\langle0\mid
e_{i,i+1}^{-}(u)=0$ for $i=1,\cdots,n-1$. The rest of the arguments
is basically the same with some adjustments. For instance, we use the expansion
\begin{align}\label{e:6.19}
\langle0\mid\sum_{\sigma}sgn(\sigma) l_{\sigma(1),1}^{-}(u)\cdots
l_{\sigma(i),i}^{-}(u+(i-1)h).
\end{align}
to calculate the coefficient of $u^{0}$ in the power
series $\langle0\mid L^{-}(u)_{1\cdots i}^{1\cdots i}$.
Similarly, we get
\begin{align}
\langle0\mid l_{i+1,i}^{(-1)}=0\quad and \quad \langle0\mid
l_{i,i}^{(-1)}=-\frac{1}{h}(d_{i}-d_{i-1})\langle0\mid
\end{align}
where $d_{0}=0$, $d_{i}(i\geq 1)$ denotes the coefficient of $u^{0}$
in $\varkappa_{1}^{-}(u)\cdots\varkappa_{i}^{-}(u+(i-1)h)$. Finally, we use the relations
\begin{align*}
\langle0\mid\sum_{\sigma}sgn(\sigma) l_{\sigma(1),1}^{-}(u)\cdots
l_{\sigma(j),i}^{-}(u+(i-1)h)=0
\end{align*}
summed over permutations $\sigma$ of the set
$\{1,\cdots,i-1,j\}$, and the expansion \eqref{e:6.19} instead of \eqref{e:6.13} and
\eqref{e:6.14}. With the use of
relations \eqref{e:6.15}, \eqref{e:6.16} and \eqref{e:6.17}, we can conclude that the
eigenvalue $\langle0\mid\ell_{k}(u)\mid0 \rangle$ coincides with the
image of the series $\ell_{k}(u)$ under the Harish-Chandra
homomorphism calculated in Theorem 5.1 for the specialization
\begin{align*}
l_{i}^{+}(u)=\varkappa^{+}(u)\quad and \quad
l_{i}^{-}(u)=\varkappa_{i}^{-}(u)
\end{align*}
for $i=1,\cdots,n$. This completes the proof by setting $\lambda_{i}(u)$ to
$\Lambda_{i}(u)$.

\end{proof}

\begin{remark}
The calculations in Sections 4, 5 and 6 are similar to those for the quantum affine algebras $U_{q}(\hat{\mathfrak{gl}}_n)$
\cite{FJ}. The results for the double Yangian can be viewed as a counterpart of those for the quantum affine algebras.
\end{remark}

\bigskip

\centerline{\bf Acknowledgments}

\medskip

The research is supported by
the National Natural Science Foundation of China (grant nos. 12101261 and 12171303) and the Simons Foundation (grant no. 523868).

\bigskip

\bibliographystyle{amsalpha}

\end{document}